%% file: main.tex
\documentclass[12pt,a4paper]{article}
\RequirePackage{amsthm,amsmath}
\RequirePackage[numbers]{natbib}
\usepackage{graphicx}

\theoremstyle{plain}
\usepackage{amsmath}
\usepackage{amsfonts}
\usepackage{amssymb}
\usepackage{amsthm}
\usepackage{mathrsfs}
\usepackage{graphicx}
\usepackage{tikz}
\usepackage{float}
\usepackage[utf8]{inputenc}
\usepackage[T1]{fontenc}
\usepackage[english]{babel}
\usepackage{tikz}
\usepackage{tabularx}
\usepackage{diagbox}
\usepackage{multirow}
\newtheorem{theorem}{Theorem}[section]
\newtheorem{definition}[theorem]{Definition}
\newtheorem{corollary}[theorem]{Corollary}
\newtheorem{lemma}[theorem]{Lemma}
\newtheorem{proposition}[theorem]{Proposition}
\newtheorem{remark}[theorem]{Remark}
\usepackage{bm}

\numberwithin{equation}{section}

\DeclareMathOperator*{\argmax}{arg\,max}

\DeclareMathOperator*{\dist}{dist}
\DeclareMathOperator*{\CTE}{CTE}
\DeclareMathOperator*{\CCTE}{CCTE}
\DeclareMathOperator*{\VaR}{VaR}
\DeclareMathOperator*{\Tube}{Tube}

\DeclareMathOperator*{\RMAE}{RMAE}

\DeclareMathOperator*{\Boule}{B}

\usepackage{color}
\usepackage{dsfont}
\newcommand{\opnorm}[1]{{\left\vert\kern-0.25ex\left\vert\kern-0.25ex\left\vert #1 
    \right\vert\kern-0.25ex\right\vert\kern-0.25ex\right\vert}}

\title{\textbf{Estimation of the covariate conditional tail expectation : a depth-based level set approach}}
\author{Elisabeth Armaut \thanks{\noindent \texttt{armaut@unice.fr}}\,
Roland Diel {\footnote{\texttt{diel@unice.fr}}}\,
Thomas Laloë {\footnote{\texttt{laloe@unice.fr}}} \vspace{0.2cm} \\
\textit{Université de Nice Sophia-Antipolis, Laboratoire J-A} \\
\textit{Dieudonné, Parc Valrose, 06108 Nice Cedex 02, France}}
\begin{document}
\clearpage
\thispagestyle{empty}
\maketitle
\textbf{Abstract.}  
The aim of this paper is to study the asymptotic behavior of a particular multivariate risk measure, the Covariate-Conditional-Tail-Expectation ($\CCTE$), based on a multivariate statistical depth function. Depth functions have become increasingly powerful tools in nonparametric inference for multivariate data, as they measure a degree of centrality of a point with respect to a distribution. A multivariate risks scenario is then represented by a depth-based lower level set of the risk factors, meaning that we consider a non-compact setting. More precisely, given a probability measure $P$ on $\mathbb{R}^d$ and a depth function $D(\cdot,P)$, we are interested in the $\alpha$-lower level set $\mathcal{L}_D(\alpha):=\left\lbrace z \in \mathbb{R}^d : D(z,P) \leq \alpha \right\rbrace$. First, we present a \textit{plug-in} approach in order to estimate $\mathcal{L}_D(\alpha).$
In a second part, we provide a consistent estimator of our $\CCTE$ for a general depth function with a rate of convergence, and we consider the particular case of the \textit{Mahalanobis} depth. A simulation study complements the performances of our estimator.
\\

\newpage
\tableofcontents
\newpage 
\section{Introduction}
Risk theory is a branch of statistics which mainly focuses on unlikely events in the aim of managing the degree of uncertainty of such events and/or the associated costs. For instance, in the context of risk management in financial institutions such as banks or insurance companies, adverse consequences may occur and usually mean potential large losses on a portfolio of assets (\citet{Eberlein07}). In hydrology, risk could represent dam floods (and/or failures) and the associated risk factors could be the amount of rainfall, water flow...
%the amount of water which exceeds the maximum storage level of a dam.
Broadly speaking, a risk measure can be viewed as a mapping from a set of real-valued random variables to $\mathbb{R}^d$, $d \geq 1$, and is used to determine the amount of an asset (or assets/goods) to be kept in reserve in order to
cover for unexpected losses. One of the most studied risk measure in the univariate risk theory is the Conditional-Tail-Expectation ($\CTE$) (\citet{denuit2006actuarial}). It characterizes the conditional expected loss given that the loss exceeds a critical loss threshold. Formally, given a real random variable $X$ with distribution  function $F_X$, the $\CTE$ at level $\alpha \in (0,1)$ is defined as : \begin{equation} \label{CTEuni}
    {\CTE}_\alpha(X):= \mathbb{E}[X | X > \VaR(\alpha)], 
\end{equation}
where 
$$\VaR(\alpha):=\inf \left\lbrace t \in \mathbb{R} : F_X(t) \geq \alpha \right\rbrace$$
is the well-known \textit{Value at Risk} which corresponds to the univariate quantile of order $\alpha$ of $X$.
Thus, the $\CTE$ is nothing but the mathematical description of an average loss in the worst $100(1-\alpha) \%$ risk scenario.\\

However, considering a single risk factor is restrictive, as we can easily imagine correlated risk factors that could be studied together. One possibility is to consider quantile regions of the risk factors distribution. In the univariate case, a wide panel of univariate quantiles has been reviewed in the literature. When it comes to multivariate risks, the study of multivariate quantile regions has increasingly been pursued in the last decades as a tool to model multivariate risk regions, especially those based on a multivariate distribution function (\citet{belzunce2007,dehaan1995, cousin2013}), or on a depth function (\citet{ZS2000}).
In this way, several generalizations in higher dimension of the $\CTE$ emerged in the literature. In particular, one can mention the one proposed by \citet{DBLMDP13} :
given a random vector $\bm X \in \mathbb{R}^d$, $\bm X=(X_1,\cdots,X_d) \in \mathbb{R}^d$, $d \geq 1$, with mutivariate distribution function $F_{\bm X}: \mathbb{R}^d \to [0,1]$, a generalization of the $\CTE$ in higher dimension is defined by
\begin{align} \label{CTEmulti}
    {\CTE}_{\alpha}(\bm X) &= \mathbb{E}[ \bm X | \bm X \in \mathcal{L}_{F_{\bm X}}(\alpha)] \in \mathbb{R}^d, %=\begin{pmatrix} 
 %\mathbb{E}[X_1 | \bm X \in \mathcal{L}_{F_{\bm X}}(\alpha)]\\ \notag 
 %\vdots \\
% \mathbb{E}[X_d | \bm X \in \mathcal{L}_{F_{\bm X}}(\alpha)]
% \end{pmatrix} 
\end{align}
where $$\mathcal{L}_{F_{\bm X}}(\alpha):=\left\lbrace x \in \mathbb{R}^d : F_{\bm X}(x) \geq \alpha \right\rbrace,~ \alpha \in (0,1),$$ is the $\alpha$-upper level set of $F_{\bm X}$, which is one generalization of the univariate quantile region $[\VaR_\alpha(X),+\infty)$ in dimension $d \geq 1$. %Note that for $d=1$, $\mathcal{L}_{F_{\bm X}}(\alpha)$ coincides withe (\ref{CTEmulti}) coincides with the quantile region in (\ref{CTEuni}).
\newline

Another interesting problem is to study the behavior of an expected cost $Y \in \mathbb{R}$ associated to $d \geq 1$ risk factors which are heterogeneous in nature.
In econometrics, for instance, one can be interested in an average return (which measures the performance of a portfolio for a certain period of time) with respect to $d \geq 1$ risk factors $\bm X \in \mathbb{R}^d$. On another note, one can also be interested in the impact of climate change (via $d$ risk factors) on high temperatures.
To address this, \citet*{DBLS2015} proposed studying the behavior of a covariate variable $Y$ on the level sets of the distribution of a $d$-dimensional vector of risk factors $\bm X$. More precisely, they define and estimate the multivariate Covariate-conditional-Tail-Expectation ($\CCTE$) defined by \begin{equation} \label{CCTEmulti}
 {\CCTE}_{\alpha}(Y,\bm X):=\mathbb{E}[Y|\bm X \in \mathcal{L}_{F_{\bm X}}(\alpha)], ~ \alpha \in (0,1).
\end{equation}

However, this $\CCTE$ based on the distribution function only considers canonical directions. For instance, it could consider an average cost associated to high or low temperatures, but not to high and low temperatures at the same time. Therefore, instead of studying the level sets $\mathcal{L}_{F_{\bm X}}(\alpha)$, \citet*{TRDiB20} studied the level sets $\mathcal{L}_{F_{R\bm X}}(\alpha)$ of a rotation $R$ of the distribution. In other words, oriented orthant are considered in order to investigate other risk regions. We propose here a more general approach, consisting in replacing the distribution function by a depth function (see \citet*{ZS2000}).\newline

Roughly speaking, a depth function is a mapping $$ D : \mathbb{R}^d \times \mathscr{P}(\mathbb{R}^d) \to \mathbb{R}_+$$
which provides a $P_{\bm X}$-based \textit{center}-outward ordering of points in $ \mathbb{R}^d,$ where $ \mathscr{P}(\mathbb{R}^d)$ denotes the set of all probability measures on $\mathbb{R}^d.$ Thus, in order to deal with risk regions, we will consider the $lower$-level sets of a depth function and propose a depth-based $\CCTE$ defined by :
\begin{align}\label{CCTE}
    {\CCTE}_{D,\alpha}(Y,\bm X):=\mathbb{E}[Y|\bm X \in \mathcal{L}_D(\alpha)],~\alpha>0,
\end{align}
where $\mathcal{L}_D(\alpha) = \left\lbrace x \in \mathbb{R}^d : D(x,P_{\bm X}) \leq \alpha \right\rbrace$ is the $\alpha$-depth based lower level set.\\

The paper is organized as follows. In Section \ref{sectiondefnot}, we introduce some notations, and tools and the mathematical definition of a depth function. Section \ref{mainresults} is devoted to our main results : in Section \ref{general CCTE section}, a construction and consistency and convergence rates of an estimator of our $\CCTE_D$ are given in a general setting, in Section \ref{Hausdorff section} we study the general asymptotic behavior of our estimator of the level set $\mathcal{L}_D(\alpha)$, and in Section \ref{MHD section} we provide consistency results and convergence rates of the $\CCTE_D$ in the particular case of $Mahalanobis$ depth. Illustrations and simulations are presented in Section \ref{SectionSimuIllu}. %Section \ref{SectionSimudata} is dedicated to simulated data while in Section \ref{Section real data} a real example is studied. 
Finally, proofs  are postponed to section \ref{proofs section}.
\section{Notations and definitions} \label{sectiondefnot}

This section is dedicated to introducing some useful notations and tools.
\subsection{General Notations}

Let $\varnothing$ be the empty set, $\mathbb{N}^*=\mathbb{N} \backslash \left\lbrace 0 \right\rbrace$ be the set of positive integers, and $\mathscr{P}:=\mathscr{P}(\mathbb{R}^d)$ be the set of all probability measures on $\mathbb{R}^d,$ $d \geq 1.$ When dealing with random variables, we assume that they are defined on a common underlying probability space $(\Omega, \mathscr{A},\mathbb{P}).$% We denote $\bm X := (X_1, X_2, \hdots, X_d)$ a random vector with distribution an element of $\mathscr{P}$.  
Given a r.v $\bm X$ with distribution $P_{ \bm X}:= P \in \mathscr{P}$ and an i.i.d sample $S_n:=(\bm{X_i})_{1 \leq i \leq n}$ of size $n \in \mathbb{N}^*$ with distribution $P$ and independent of $\bm X$, we denote by $P_n := \sum_{i=1}^n\delta_{\bm{X_i}}$ the empirical measure based on this finite sample. For notational convenience, we denote by $\mathbb{E}_{P}$ the mathematical expectation under $P$, and by $\mathbb{E}_{S_n}[Z]:=\mathbb{E}[Z | \bm X_1, \cdots, \bm X_n]$ the conditional expectation of $Z$ knowing $ \bm X_1, \cdots, \bm X_n$. Moreover, denoting $\Phi(A) := P_{\bm X}[A]=P[A]$ for any Borel-set $A \subset \mathbb{R}^d$, we denote by $P_{S_{n}}[A] := \Phi[A(X_1,\cdots,X_n)]$, where $A := A(X_1,\cdots,X_n)$ is a subset of $\mathbb{R}^d$ which depends on the data $X_1,\cdots,X_n$. Thus, $A(X_1,\cdots,X_n)$ is a random subset, so that $P_{S_{n}}[A]$ is a r.v. Furthermore, for any real number $q >0$, let $ \mathbb{L}^q(\Omega):=\mathbb{L}^q(\Omega,\mathscr{A},\mathbb{P})$ denote the vector space of real-valued random variables $U$ for which $\mathbb{E}[|U|^q] < +\infty$. \\

Unambiguously, we denote by $\| \cdot \|$ the Euclidean norm in $\mathbb{R}^d$ and by $\opnorm{\cdot}$ the matrix norm induced by the Euclidean norm in $\mathbb{R}^d,$ i.e for any $d \times d$ real matrix $M$, $$ \opnorm{M}:= \sup_{x \neq 0} \frac{\|M x \|}{\|x\|}=\sup_{\|x\|=1} \|M x \|.$$ 
%We also denote by $\Boule(x,r)$ the (open) ball centered on $x \in \mathbb{R}^d$
For $p \in [1,+\infty]$, we also denote by \begin{align*}
& \| f \|_{p,\lambda} := \left(\int_{\mathbb{R}^d} |f(x)|^pdx\right)^{\frac{1}{p}} \text{~~for~}  p < +\infty, \text{~and}  \\
&  \|f\|_{\infty,\mathbb{R}^d} :=\text{ess sup}_{x \in \mathbb{R}^d}|f(x)| \text{~~for~}  p=+\infty,
\end{align*} the $\mathbb{L}^p(\mathbb{R}^d,\lambda_d)$ norm of $f$ w.r.t the Lebesgue measure on $\mathbb{R}^d$.\\

We recall that for A and B non-empty compact sets in $(\mathbb{R}^d,\| \cdot \|)$ the Hausdorff distance between A and B is defined by \begin{center}
    $d_H(A,B)=\displaystyle \sup \left( \sup_{a \in A} \dist(a,B), \sup_{b \in B} \dist(b,A) \right)$,
\end{center}
where, $$\dist(x,A):=\displaystyle \inf_{a \in A}:= \|a-x\|.$$

Now, given two sequences of real numbers $(u_n)_{n \geq 1}$ and $(v_n)_{n \geq 1}$, we recall that $u_n = O_{n}(v_n)$ means that there exist a constant $C>0$ and $N \in \mathbb{N}^*$ s.t. for all $n \geq N,$ $|u_n| \leq C |v_n|$. \\

Finally, let $(X_{n})_{n \in (\mathbb{N}^{*})^r}$, $r \geq 1$, be a set of random variables and $(u_{n})_{n \in (\mathbb{N}^{*})^r}$ be a deterministic set of positive real numbers, we recall the following classical stochastic dominance \\

$\bm{ X_n = \mathcal{O}_{\mathbb{P},n}  (u_{n})  \stackrel{\text{def}}{\Longleftrightarrow}}$
\begin{align*}
 \forall~ \varepsilon>0, & ~ \exists ~ M_\varepsilon >0, ~ \exists~ N_\varepsilon \geq 1,~\forall ~ n:=(n_1,\cdots,n_r) \in (\mathbb{N}^{*})^r, \\
  & \min_{ 1 \leq i \leq r} n_i \geq N_\varepsilon \Rightarrow \mathbb{P}[X_{n} \geq M_\varepsilon \cdot u_{n}] \leq \varepsilon.
\end{align*}
\subsection{Depth functions}
In this section, we formally introduce the definition of a statistical multivariate depth function as in \citet*{ZS2000} (Definition 2.1 in \citep{ZS2000}). 
Let us begin with some common multivariate symmetry notions which have been widely used in the literature (the interested reader can refer to \citet{ZS2000}, \citet{liu1990}, and \citet{Beran1997}). In what follows, we review some standard symmetry notions.\\

$\bullet$ $C$-symmetry : a random vector $\bm X \in \mathbb{R}^d$ is \textit{centrally-symmetric} (or $C$-symmetric) about $\theta \in \mathbb{R}^d$ if $\bm X-\theta \stackrel{\rm d}{=} \theta - \bm X$.\\

$\bullet$ $A$-symmetry (Liu, 1990) : $\bm X$ is said to be \textit{angularly-symmetric} (or $A$-symmetric) about $\theta$, if $(\bm X-\theta)/ \| \bm X - \theta \|$ is centrally symmetric about the origin.\\

$\bullet$ $H$-symmetry (Zuo and Serfling, 2000) : $\bm X$ is said to be \textit{halfspace-symmetric} (or $H$-symmetric) about $\theta$ if $P_{\bm X}[H] \geq 1/2$ for every closed halfspace $H$ containing $\theta$. \\

Note that, it is easily established that $C$-symmetry $\Rightarrow$ $A$-symmetry $\Rightarrow$ $H$-symmetry.
In a natural terminology, $\theta$ is called the \textit{center} of the distribution $P_{\bm X}$.

\begin{definition}[\citet{ZS2000}] \label{defprofmulti}
A statistical depth function is a mapping $D: \mathbb{R}^d \times \mathscr{P} \to \mathbb{R}$ which is bounded, non negative, measurable in its first argument and satisfying :\\
\textbf{(D1) Affine invariance} : for any $P_{\bm X} \in \mathscr{P}$, $b \in \mathbb{R}^d$, and any invertible size $d$ matrix $A$, $D(Ax+b,P_{A \bm X+b})=D(x,P_{\bm X})$\\
\textbf{(D2) Maximality at center} : for any $P_{\bm X} \in  \mathscr{P}$ having a unique \textit{center} $\theta \in \mathbb{R}^d$ (for one of the symmetry notions previously presented), $D(\theta,P_{\bm X})= \sup_{x \in \mathbb{R}^d}D(x,P_{\bm X})$\\
\textbf{(D3) Monotonicity relative to deepest point} :  for any $P_{\bm X}$ having deepest point $\theta$ i.e. $D(\theta,P_{\bm X})=\sup_{x \in \mathbb{R}^d}D(x,P_{\bm X})$, $D(x,P_{\bm X}) \leq D(\alpha x + (1-\alpha)\theta,P_{\bm X})$ holds for $\alpha \in [0,1]$\\
\textbf{(D4) Vanishing at infinity :} $D(x,P_{\bm X}) \to 0$ as $\|x\| \to \infty$, for each $P_{\bm X} \in \mathscr{P}$.
\end{definition}

Informally, the first property of a depth $\bm{(D1)}$ suggests that the depth of a point $x \in \mathbb{R}^d$ does not depend on the underlying coordinate system. As far as property $\bm{(D2)}$ is concerned, for a distribution having a unique "center" i.e., the point of symmetry with respect to some notion of multivariate symmetry, the depth function should attain its maximum value at this center. % where $ \stackrel{d}{=}$ denotes "equal in distribution".
Property $\bm{(D3)}$ illustrates the fact that as a point $x \in \mathbb{R}^d$ moves away from the point of maximal depth (for instance the "center" of a distribution) along any fixed ray through the center, the depth at $x$ should decrease monotonically. Last but not least, property $\bm{(D4)}$ implies that the depth of a point $x$ approaches zero as $\|x\|$ approaches infinity.
Note that $\bm{(D3)}$ and $\bm{(D4)}$ mean that the upper level sets $$\left\lbrace  x \in \mathbb{R}^d : D(x,P_{\bm X}) \geq \alpha \right\rbrace,~\alpha >0,$$ are bounded and starshaped about the point of maximum depth.

\section{Main results} \label{mainresults}
In this section, we define a risk measure based on a general depth function, the Covariate-Conditional-Tail-Expectation ($\CCTE_D$) and we propose an estimator of the $\CCTE_D$ using a plug-in estimator of the level set. We study the asymptotic behavior of the $\CCTE_D$ when consistency of the level sets in terms of the probability under $P \in \mathscr{P}$ of the volume of the symmetric difference is provided.
\subsection{General Covariate-Conditional-Tail-Expectation consistency} \label{general CCTE section}
Fix a depth function $D: \mathbb{R}^d \times \mathscr{P} \to \mathbb{R}$ and a distribution $P \in \mathscr{P}$. We denote $$\alpha_{\max}(P):=\sup_{z \in \mathbb{R}^d} D(z,P)=\sup_{z \in \mathbb{R}^d} D(z).$$
Consider a couple $(Y, \bm X)$ s.t. $Y$ is a real random variable which is dependent on a random vector $\bm X \in \mathbb{R}^d$ with ditribution $P$.
In Definition \ref{defCCTE}, we formally define our $\CCTE_D$ and propose an estimator of the latter. For $n_1,~n_2 \geq 1$, let
 \begin{align} \notag
    \tilde{S}_{n_1}:=(\tilde{\bm X}_i)_{i=1,..,n_1} \text{~be an i.i.d~} n_1 \text{-sample from~} P, \text{~and} \\ \notag
    S_{n_2}:=((Y_i,\bm X_i))_{i=1,..,n_2} \text{~be an i.i.d~} n_2 \text{-sample from~} P_{(Y,\bm X)},
\end{align}
s.t. $\tilde{S}_{n_1}$ and $S_{n_2}$ are independent.
Now, we define the $\alpha$-\textit{lower level set} of $D$ and its plug-in estimator based on $\tilde{S}_{n_1}$ by 
\begin{align*}
  & \bm{ \mathcal{L}_D(\alpha)}=\bm{ \mathcal{L}_D(\alpha,P):=\left\lbrace  x \in \mathbb{R}^d : D(x,P) \leq \alpha \right\rbrace,} \text{and}\\
  & \bm{\mathcal{L}_{n_1}(\alpha):=\bm{\mathcal{L}_{D}(\alpha,\tilde{P}_{n_1})}=\left\lbrace  x \in \mathbb{R}^d : D_{n_1}(x):=D(x,\tilde{P}_{n_1}) \leq \alpha \right\rbrace},
\end{align*}
where $\tilde{P}_{n_1}$ is the empirical measure based on the sample $\tilde{S}_{n_1}:=(\bm{\tilde{X}_i})_{1 \leq i \leq n_1}$. \\
Finally, we provide the definition of our $\CCTE_D$ and its associated estimator.
\begin{definition}[Depth-based Covariate-Conditional-Tail-Expectation] \label{defCCTE}
Let $\bm X$ $\in \mathbb{R}^d$ be a random vector with distribution $P \in \mathscr{P}$ and $Y$ be an integrable real random variable (which is dependent on $\bm X$). Let $\alpha>0$ and assume $P[\mathcal{L}_D(\alpha)]>0$. \vspace{0.2cm}\newline
$(i)$ The depth-based Covariate-Conditionale-Tail-Expectation at level $\alpha$ is defined by : $${\CCTE}_{D,\alpha}(Y,\bm X):=\mathbb{E}[Y|\bm X \in \mathcal{L}_D(\alpha)].$$
$(ii)$ Its estimator based on the sample $S_{n_2}$ is given by :
\begin{align}\label{estmiateurCCTE}
     \widehat{\CCTE}_{D,\alpha}^{n_1,n_2}(Y,\bm X)  := \displaystyle \frac{ \displaystyle \sum_{i=1}^{n_2}Y_i \mathds{1}_{\bm X_i \in \mathcal{L}_{n_1}(\alpha)}}{\displaystyle  \sum_{i=1}^{n_2}\mathds{1}_{\bm X_i \in \mathcal{L}_{n_1}(\alpha)}},
\end{align}
with the convention $0/0 = 0.$
\end{definition}
%sous réserve que, pour $n_1,n_2 \geq 1$, $\mathbb{P}$-p.s. $\sum_{i=1}^{n_2}\mathds{1}_{\bm X_i \in \mathcal{L}_{n_1}(\alpha)}> 0,$ (ou de façon équivalente, si pour $n_1,n_2 \geq 1$, $\mathbb{P}$-p.s, il exite au moins un $1 \leq i \leq n_2$, tel que $\bm X_i$ est dans $\mathcal{L}_{n_1}(\alpha)$).
Our first result, namely Theorem \ref{vitesseCCTEproba}, links the rate of convergence of the $\CCTE_D$ to the one of the symmetric difference between the true and estimated $\alpha$-level set. We first state the following assumption describing a convergence rate for the level sets:\vspace{0.2cm}\newline
 \textbf{(H0):} there exists an increasing sequence of positive real numbers $(v_{n_1})_{n_1 \geq 1}$ s.t. 
    \begin{equation*}
       P_{\tilde{S}_{n_1}}[\mathcal{L}_{n_1}(\alpha) \Delta \mathcal{L}_D(\alpha)] = \mathcal{O}_{P,n_1}\left(v_{n_1}^{-1}\right),
    \end{equation*}
where $A\Delta B=(A\backslash B)\cup (B\backslash A)$ is the symmetric difference between $A$ and $B$. \\
In the spirit of \citet{DBLS2015}, Theorem \ref{vitesseCCTEproba} states that, under some conditions, the $\CCTE_D$ estimator is consistent with at most a convergence rate $O(\sqrt{n_2}).$ Remark that in Theorem \ref{vitesseCCTEproba}, the $r$-\textit{th} moment of $Y$ is only involved in the rate $(v_{n_1})$. Note that in our setting, the boundaries of the depth-based level sets at hand are compact, contrary to the non-compact setting studied in \citet{DBLS2015}.
\begin{theorem} \label{vitesseCCTEproba}
Let $\alpha >0$ and $P \in \mathscr{P}$. Assume
 $P[\mathcal{L}_D(\alpha)]>0 $, and $\textbf{(H0)}$ is satisfied and there exists $r \in [2,\infty]$ s.t.  $Y \in \mathbb{L}^r(\Omega).$
Then, it holds that
\begin{align*}
  \big| \widehat{\CCTE}_{D,\alpha}^{n_1,n_2}(Y,\bm X) -  {\CCTE}_{D,\alpha}(Y,\bm X)  \big| = \mathcal{O}_{P,n_1,n_2}\left(n_2^{-\frac{1}{2}} \vee v_{n_1}^{-(1-\frac{1}{r})}\right).
\end{align*}
\end{theorem}
%Note that the convergence rates are obtained regardless of the order between $n_1$ and $n_2$ when $n_1,n_2$ tend to infinity.
 Furthermore, following the approach of \citet{DBLS2015}, Assumption \textbf{(H0)} can be replaced by Assumption \textbf{(H1)} and one can derive a similar result to Theorem \ref{vitesseCCTEproba} (c.f. Corollary \ref{corSymDiff}) : \vspace{0.1cm}\newline
\textbf{(H1):} $(i)$ there exists an increasing sequence of positive real numbers $(v_{n_1})_{n_1 \geq 1}$ s.t. \begin{equation*}
     \lambda_d(\mathcal{L}_{n_1}(\alpha)\Delta \mathcal{L}_D(\alpha))=\mathcal{O}_{P,n_1}\left(v_{n_1}^{-1}\right), \text{~and}
 \end{equation*}
\hspace{1cm}$(ii)$ $P$ is absolutely continuous with density function $f \in \mathbb{L}^p(\mathbb{R}^d,\lambda_d)$ for some $p \in (1,+\infty].$
\begin{corollary}\label{corSymDiff}
Let $\alpha >0$ and $P \in \mathscr{P}$. Assume that
 $P[\mathcal{L}_D(\alpha)]>0,$ and that there exists $r \in [2,+\infty]$ s.t. $Y \in \mathbb{L}^r(\Omega).$
Let $(v_{n_1})_{n_1}$ satisfy \textbf{(H1)}, then
 \begin{align*}
 \big| \widehat{\CCTE}_{D,\alpha}^{n_1,n_2}(Y,\bm X) -  {\CCTE}_{D,\alpha}(Y,\bm X)  & \big| = \mathcal{O}_{P,n_1,n_2}\left(n_2^{-\frac{1}{2}}\vee v_{n_1}^{-\left(1-\frac{1}{r}\right)\left(1-\frac{1}{p}\right)}\right).
\end{align*} 
\end{corollary}
\begin{proof}
It is sufficient to show that under the assumptions of Corollary \ref{corSymDiff}, assumption \textbf{(H0)} of Theorem \ref{vitesseCCTEproba} is satisfied by the sequence $(v_{n_1}^{1-1/p})_{n_1}$. When $ p \in (1,+\infty),$ it holds almost-surely
\begin{align*}
   v_{n_1}^{1-\frac{1}{p}}P_{\tilde{S}_{n_1}}[ \mathcal{L}_{n_1}(\alpha) \Delta \mathcal{L}_D(\alpha)] & =   v_{n_1}^{1-\frac{1}{p}} \int \mathds{1}_{\bm x \in \mathcal{L}_{n_1}(\alpha) \Delta \mathcal{L}_D(\alpha)}f(\bm x) d\bm x \\
   & \leq    v_{n_1}^{1-\frac{1}{p}} \lambda_d( \mathcal{L}_D(\alpha) \Delta \mathcal{L}_{n_1}(\alpha))^{1-\frac{1}{p}} \| f \|_{p,\lambda} \text{~~(Hölder)}\\
    & \leq  \underbrace{(v_{n_1} \lambda_d( \mathcal{L}_D(\alpha) \Delta \mathcal{L}_{n_1}(\alpha))^{1-\frac{1}{p}}}_\textrm{ $\mathcal{O}_{P,n_1}(1)$ (\textbf{H1})(i)} \underbrace{\| f \|_{p,\lambda}}_\textrm{$<\infty$ (\textbf{H1})(ii)}.
\end{align*}
When $p= +\infty$, the result is trivially valid by bounding $f$ by its essential supremum.
\end{proof}
\subsection{Consistency of general depth-based level sets in terms of the Hausdorff distance} \label{Hausdorff section}

In this section, the problem of interest is to study the conditions under which assumption \textbf{(H1)(i)} is satisfied : that would provide a rate of convergence for the general $\CCTE_D$. This means studying the rate of convergence of the volume of the symmetric difference between $\bm{\mathcal{L}_n}$ and $\bm{\mathcal{L}_{D}}$. It happens that, by controlling the Hausdorff distance between the respective boundaries  $\bm{\partial \mathcal{L}_D(\alpha)}$ and $ \bm{\partial \mathcal{L}_{n}(\alpha)}$ of those two sets, one can control the volume of the symmetric difference. Remark first that the Hausdorff distance between those sets is asymptotically well defined since the boundaries of the level sets are compact and non empty (c.f. Remark \ref{remarquebordalphamax} together with Remark \ref{remark2haussdistance}).
%Thus, we study the Hausdorff distance between the boundaries of the level sets, i.e. $\bm{\partial \mathcal{L}_D(\alpha)}$ and $ \bm{\partial \mathcal{L}_{n}(\alpha)}$, as a criteria of consistency. But first, we have to check that the Hausdorff distance between those sets is well defined. As a matter of fact, the Hausdorff distance $d_H(A,B)$ between $A$ and $B$ is well defined when $A$ and $B$ are non-empty closed sets (not necesssarily compact) but in this case the value $d_H(A,B)$ could be infinity. In our setting, the latter situation is not dealt with as the boundaries of the level sets are compact (c.f. Remark \ref{remarquebordalphamax}).
\begin{remark}\label{remarquebordalphamax}
On one hand, if $\alpha \in (0, \alpha_{\max}(P))$, the sets $\left\lbrace x : D(x,P) \geq \alpha \right\rbrace$ and $\left\lbrace x : D(x,P) < \alpha \right\rbrace$ are both non-empty, and since $D$ is vanishing at infinity, the empirical level set $\left\lbrace x : D(x,P_n) < \alpha \right\rbrace$, $n \geq 1$,  is non empty. Thus, $\mathcal{L}_D(\alpha)$ and its boundary $\partial \mathcal{L}_D(\alpha)$ are non-empty. On the other hand, if the empirical depth a.s. (almost-surely) converges pointwise to its true version on $\mathbb{R}^d$ then (cf. Theorem 4.1 in \citet{dyckerhoff2017}), for any $\alpha \in (0, \alpha_{\max}(P))$ and a.s. for any $n$ large enough, the upper-level set $\left\lbrace x : D(x,P_n) \geq \alpha \right\rbrace$
%(whose levels depend on $\alpha_{\max}(P_n)$) 
is non-empty as well. Thus, a.s, for $n$ large enough, $\partial \mathcal{L}_n(\alpha)$ is non-empty.
\end{remark}
Proposition \ref{PropCasalRodriguez} is a slight modification of the Proposition 3.1 in the Ph.D. thesis of \citet{RodCasal03} adapted to depth functions. 
%Elle permet, sous certaines conditions, d'assurer la validité de la seconde hypothèse du Théorème \ref{theoVitessedH}: celle-ci met en évidence le caractère localement lipschitzien de l'application $\tilde{\alpha} \mapsto \left\lbrace D= \tilde{\alpha} \right\rbrace$ au voisinage du niveau $\alpha.$
We introduce the following assumption \textbf{(L)} which characterizes the locally Lipschitz behavior of the mapping $\tilde{\alpha} \mapsto \left\lbrace D= \tilde{\alpha} \right\rbrace$ w.r.t the Hausdorff distance in a neighborhood of the fixed level $\alpha>0$ . \vspace{0.1cm}\newline
$(\textbf{L}):$
$\exists~A>0,~\exists~\gamma>0,~\forall ~ \beta>0, |\alpha-\beta| \leq \gamma \Rightarrow d_H( \left\lbrace D=\alpha \right\rbrace,  \left\lbrace D=\beta \right\rbrace) \leq \displaystyle A |\alpha-\beta|.$
\begin{proposition}\label{PropCasalRodriguez}
Let $D:\mathbb{R}^d \times \mathscr{P} \to \mathbb{R}_+$ be a multivariate depth function. Let $\alpha>0$, $0<\varepsilon<\alpha$ and $P \in \mathscr{P}$ be fixed. Denoting $D(x):=D(x,P)$, assume that \\
(i) the function $x \mapsto D(x)$ is of class $\mathscr{C}^2$ on the set $\mathcal{K}_\varepsilon(\alpha):=D^{-1}([\alpha-\varepsilon,\alpha+\varepsilon]),$ and \newline
(ii) $\displaystyle m_\nabla:=m_\nabla(\alpha,\varepsilon,P):=\inf_{x \in \mathcal{K}_\varepsilon(\alpha)}\|(\nabla D)_x\| >0$, where $(\nabla D)_x$ is the gradient of $D(\cdot)$ at x.\vspace{0.1cm}\newline
Then D satisfies Assumption \textbf{(L)}, with $A=\frac{2}{m_\nabla}.$
\end{proposition}

The following result is an adapted version of Theorem 2 in \citet{CuevasRodri06} to depth functions, where we weaken the assumption of continuity of the empirical depth function by an assumption of upper-semicontinuity.
\begin{theorem}\label{theoVitessedH}
Let $D:\mathbb{R}^d \times \mathscr{P} \to \mathbb{R}_+$ be a depth function, $P \in \mathscr{P}$ and $\alpha \in (0,\alpha_{\max}(P))$. Denote by $P_n$ an estimator of $P$, $n \geq 1$. Suppose that $x \in \mathbb{R}^d \mapsto D(x,P):=D(x)$ is a continuous function, and $x \in \mathbb{R}^d \mapsto D_n(x):=D(x,P_n)$ is upper semi-continuous $\mathbb{P}$-almost surely for any $n \geq 1$, and that $$\|D_n- D\|_{\infty,\mathbb{R}^d} \xrightarrow[n \to \infty]{a.s} 0.$$ Under the same assumptions as in Proposition \ref{PropCasalRodriguez}, it holds that $$d_H(\partial \mathcal{L}(\alpha),\partial \mathcal{L}_n(\alpha))=\underset{n \to \infty}{O}(\|D_n- D\|_{\infty,\mathbb{R}^d}),~\mathbb{P}\text{-a.s}.$$
\end{theorem}
\begin{remark} \label{remark2haussdistance}
Recall that, for $n$ large enough, $\partial \mathcal{L}_n(\alpha) \neq \varnothing$  $\mathbb{P}$-a.s, so that the Hausdorff distance is well-defined for large $n$ (see Remark \ref{remarquebordalphamax} as well). 
Indeed, in Theorem \ref{theoVitessedH}, one can underline two main properties of a depth function, namely : upper semi-continuity which is equivalent to having closed depth-based upper level sets, and Property \textbf{(D4)} (vanishing at infinity) which guarantees that the upper level sets are bounded. As a consequence, the Hausdorff distance is well defined since $\partial \mathcal{L}_n(\alpha)$ is closed by definition and is bounded as it is included in the compact $\alpha$-upper level set ($\alpha>0$). The same applies for $\partial \mathcal{L}(\alpha)$ (or immediately by continuity of D). What is more, $\alpha \in  (0,\alpha_{\max}(P))$ implies $\partial \mathcal{L}(\alpha) \neq \varnothing$ (cf. Remark \ref{remarquebordalphamax}). 
\end{remark}
%%D[D[-(1/theta)*log(1+ (exp(-theta*exp(-exp((m-x)/beta))))*(exp(-theta*exp(-exp((p-y)/gamma))))/(exp(-theta)-1)),{x,1}],{y,1}]
\subsection{\textit{MHD}-depth based Covariate-Conditional-Tail-\\Expectation consistency}\label{MHD section}
The \textit{Mahalanobis} depth function (Example 2.5 in \citet{ZS2000}) %has all the desired properties in order to satisfy all the results of the previous sections. It 
is a depth function in the sense of Definition \ref{defprofmulti} (see Definition \ref{MHDDepth}), and is smooth as a function of $x$ (which implies the upper-semicontinuity property in the empirical case as well). \newline %It satisfies property \textbf{(D4)} (vanishing at infinity), it is uniformly consistent on $\mathbb{R}^d$ and, mostly, it is a smooth function as a function of $x$ (which implies the upper-semicontinuity property in the empirical case as well).\\

In order to study the rate of convergence of the $\CCTE_D$ based on $MHD$, we check here Assumption \textbf{(H1)(i)}. According to Section \ref{Hausdorff section}, the problem reduces to studying the rate of convergence of $\|D_n-D\|_\infty$, in probability, when $D=MHD$ (c.f. Section \ref{Hausdorff section}, Theorem \ref{theoVitessedH}).
\begin{definition}[\textit{Mahalanobis depth}, \citet{ZS2000}] \label{MHDDepth}  
Let $\bm X \in \mathbb{R}^d$ be a random vector with distribution $P \in \mathscr{P}$. The \textit{Mahalanobis} depth is defined by 
 $$
MHD(x,P)= \left\{
    \begin{array}{ll}
        \displaystyle \left(1+ d_{\Sigma_{\bm X}}^2(x, \mu_{\bm X}) \right)^{-1} & \mbox{if }\mathbb{E}_P[\| \bm X\|^2] < +\infty \\
        0 & \mbox{if } \mathbb{E}_P[\| \bm X\|^2] =+\infty
    \end{array}
\right.
$$ 
where $\mu_{\bm X}= \mathbb{E}_P[\bm X]$ is the mean vector of $\bm X$ and $\Sigma_{\bm X} $ is its covariance matrix (which is assumed to be invertible) and $$d_{\Sigma_{\bm X}}^2(x, \mu_{\bm X}):=\|x - \mu_{\bm X} \|_{\Sigma_{\bm X}}^2:=~^t(x-\mu_{\bm X})\Sigma_{\bm X}^{-1}(x-\mu_{\bm X})$$ is the \textit{Mahalanobis} distance. 
\end{definition}
\begin{remark}
Note that, the above definition of $MHD$ depth is introduced as such in order to highlight the fact that, it is restricted to ditributions with second moment while still remaining a depth function in the sense of Definition \ref{defprofmulti}. Furthermore, for a fixed distribution $P$, the function $x \in \mathbb{R}^d \mapsto MHD(x,P)$ is infinitely differentiable, concave, and has $x=\mu_{\bm X}$ as unique critical point, thus $\displaystyle \mu_{\bm X}=\argmax_{x \in \mathbb{R}^d}MHD(x,P)$. And $\displaystyle \alpha_{\max}(P):=\max_{x \in \mathbb{R}^d}MHD(x,P)=1.$
\end{remark}
In \citet{ZS2000}, one can also find the following result which underlines the properties of \textit{MHD} as a depth function (Theorem 2.10, \citet{ZS2000}).
\begin{proposition}
[\citet{ZS2000}]\label{propdeprofondeurMHD}
Let $\bm X \in \mathbb{R}^d$ be a random vector with distribution $P \in \mathscr{P}$. %s.t. $\mathbb{E}_P[\| \bm X\|^2] < \infty$. 
Assume $\bm X$ is symmetric (for some notion of symmetry) about $\theta \in \mathbb{R}^d$. If $\mu_{\bm X}=\theta$, then $MHD$ is a statistical depth function in the sense of Definition \ref{defprofmulti}.
\end{proposition}
%\begin{table}[H]
%\noindent
%\noindent \begin{tabularx}{\linewidth}{|c||X|X|X|X|X|X|}
%\hline
%\multicolumn{7}{c}{ \textbf{Consistance faible}} \\
%\hline
%  & \textbf{D1} & \textbf{D2} & \textbf{D3} & \textbf{D3conv} & \textbf{D4} & \textbf{D5} \\
%\hline
%$MHD$ & \centering \textcolor{black}{$\checkmark$} & \centering %\textcolor{black}{$\checkmark$} & \centering \textcolor{black}{$\checkmark$} & \centering %\textcolor{black}{$\checkmark$} & \centering \textcolor{black}{$\checkmark$} & \centering %\textcolor{black}{$\checkmark$} \tabularnewline
% \hline
%\end{tabularx}
%\caption{\label{profmultitable} Properties of \textit{MHD} as a depth function.}
%\end{table}
\begin{figure}[H]
    \centering
    \includegraphics[scale=0.4]{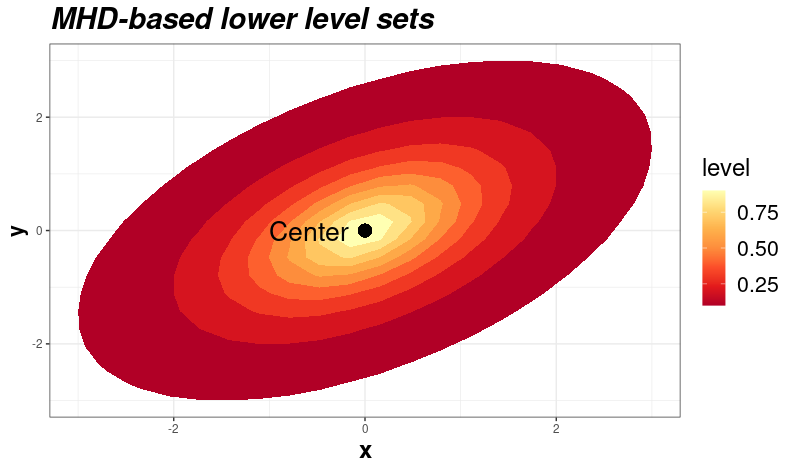}
    \caption{Theoretical lower-level sets based on $MHD(\cdot,P),$ with $P$ the law of a Gaussian vector in $\mathbb{R}^2.$}
    \label{MHDlevelsetsfig}
\end{figure}
A natural estimator of $MHD$ is given by \begin{equation}\label{estmMHD}
\displaystyle MHD_n(x):=MHD(x,P_n)=\left(1+ ~^t(x-\hat{\mu}_n)\hat{\Sigma}_n^{-1}(x-\hat{\mu}_n)) \right)^{-1},
\end{equation}
where $\hat{\mu}_n$ is the empirical mean vector and $\hat{\Sigma}_n$ is the empirical covariance matrix.

In Proposition \ref{propositiondeMHDD_h}, we provide the particular version of Proposition \ref{PropCasalRodriguez} associated to $MHD$ and $MHD_n$ depths.
\begin{proposition}\label{propositiondeMHDD_h}
Let $\bm X \in \mathbb{R}^d$ be a random vector from $P \in \mathscr{P}$ s.t. $\mathbb{E}_P[\| \bm X\|^2] < \infty.$
Consider $D(x):=MHD(x):=MHD(x,P)$ and fix $\alpha \in (0,1)$ and $0<\varepsilon<\alpha \wedge (1-\alpha)$. Denoting $\mathcal{K}_\varepsilon(\alpha):=D^{-1}([\alpha \pm \varepsilon])$, it holds \vspace{0.2cm} \newline
\textbf{(i)} $\displaystyle m_\nabla:=\inf_{x \in \mathcal{K}_\varepsilon(\alpha)} \|(\nabla D)_x\|>0,$ and \newline
\textbf{(ii)} %if $\mathbb{P}$-a.s, for all $n \geq 1,$ $\partial \mathcal{L}_n(\alpha) \neq \varnothing$, then
$d_H(\partial \mathcal{L}_{MHD}(\alpha),\partial \mathcal{L}_n(\alpha))=\underset{n \to \infty}{O}(\|MHD_n- MHD\|_{\infty,\mathbb{R}^d}),~\mathbb{P}\text{-a.s}.$
\end{proposition}
Now, Theorem \ref{dnRate_MHD} is a useful result in which we provide the rate of convergence (in probability) of $MHD_n$ to its population version $MHD$ uniformly on $\mathbb{R}^d$. %is at most $\sqrt{n.}$ 
\begin{theorem}\label{dnRate_MHD}
Let $\bm X=(X^{(1)},\cdots,X^{(d)})$ be a random vector with distribution $P \in \mathscr{P}$ satisfying $\mathbb{E}_P[|X^{(i)}|^4]<\infty$ for all $1 \leq i \leq d$. Then, it holds that \begin{center}
   $\displaystyle \| MHD_n - MHD \|_{\infty, \mathbb{R}^d} =\mathcal{O}_{P,n}\left(n^{-\frac{1}{2}}\right).$
\end{center}
\end{theorem}
Finally, in Theorem \ref{MHD_vitesseCCTE}, we derive the specific rate of convergence for the $\CCTE$ based on $MHD$-depth.
\begin{theorem} \label{MHD_vitesseCCTE}
Let $P \in \mathscr{P}$, $D(\cdot,P)=MHD(\cdot,P)$ and $\alpha \in (0,1)$. Assume $P[ \mathcal{L}_D(\alpha)] >0. $
Under the assumptions of %Proposition \ref{propositiondeMHDD_h},
Theorem \ref{dnRate_MHD}, and Assumption \textbf{(H1)(ii)} and assuming moreover that there exists $r \in [2, +\infty]$ s.t. $Y \in \mathbb{L}^r(\Omega)$, it holds that
\begin{align*}  \displaystyle \left|  \widehat{\CCTE}_{D,\alpha}^{n_1,n_2}(Y,\bm X) - {\CCTE}_\alpha(Y,\bm X)  \right| = \mathcal{O}_{P,n_1,n_2}\left(n_2^{-\frac{1}{2}}\vee n_1^{-\frac{1}{2}\left(1-\frac{1}{r}\right)\left(1-\frac{1}{p}\right)}\right).
\end{align*}
%where $d_{n_1}=o(\sqrt{n_1})$.
\end{theorem}
\section{Simulations and illustrations} \label{SectionSimuIllu}
%\subsection{$\CCTE_{MHD,\alpha}(Y,\bm X)$ estimation via copulas} \label{SectionSimudata}
In this section, we provide an illustration of Theorem \ref{MHD_vitesseCCTE}. We study the estimated  $\CCTE_D$ for cost variables $Y$ which are dependent on the law of $\bm X := (X_1, X_2) \in \mathbb{R}^2$ and having the form : $$ Y= \| \bm X\|^2 + \varepsilon, $$ where $\varepsilon \sim \mathcal{N}(0,\sigma^2),~ \sigma^2 >0,$ is a gaussian noise. In our simulations we will take $\sigma^2=0.005.$ Here, we choose the squared euclidian norm which has fourth moment under $P$, defined by : $\| \bm x \|^2=|x_1|^2+|x_2|^2$ (see Figure \ref{normcostfunction}).
Moreover, we consider dependent risk factors $X_1$ and $X_2$ via a bivariate Frank Copula with Gumbel marginals with parameter $(\mu, \beta)=(0,0.25)$ and $(-0.5,0.25)$ respectively (Figure \ref{samplefig}).
\begin{figure}[H]
    \centering
    \includegraphics[scale=0.45,width=7.5cm,height=6.5cm]{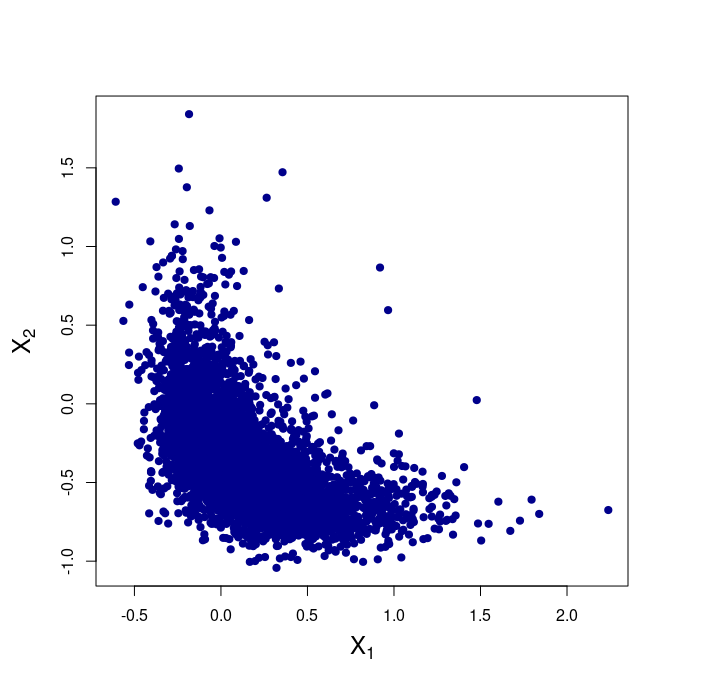}
    \caption{Sample of dependent Gumbel marginals via a Frank Copula.} \label{samplefig}
\end{figure}
 %However, in practice, regression functions can be
  Note that the above example satisfies the assumptions of Theorem \ref{MHD_vitesseCCTE}. 
 \begin{figure}[H]
    \centering
    \includegraphics[scale=0.5]{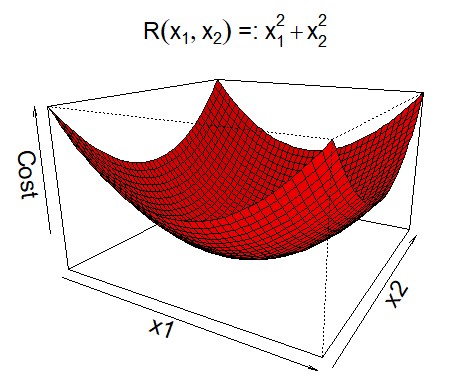}
    \caption{\textit{Monotonic} function : $R(x_1,x_2) = |x_1|^2+|x_2|^2.$}
    \label{normcostfunction}
\end{figure} 
Here we compare $\widehat{\CCTE}_{MHD,\alpha}^{n_1,n_2}$ with the theoretical $\CCTE_{MHD,\alpha}$ for \textit{Mahalanobis} depth. For the sake of simplicity, we take $n_1=n_2=n$. Assuming that $p,r \approx + \infty$ with our sample, we obtain that $|\widehat{\CCTE}_{MHD,\alpha}^n-\CCTE_{MHD,\alpha}|$ decays to zero at most with a convergence rate $o(\sqrt{n}).$\\
 
We provide a deterministic approximation of the true mean vector and covariance matrix, $m$ and $\Sigma$ respectively. However, due to the complexity of the level-sets as domains of integration in the computation of the $\CCTE_D$, we perform a Monte Carlo procedure to fix the "true" value of the $\CCTE_D$ based on a sample of size $10^8$ (without noise), that is:
$$ \displaystyle \frac{ \displaystyle \sum_{i=1}^{10^8}R(\bm X_i) \mathds{1}_{\bm X_i \in \mathcal{L}_{MHD}(\alpha)}}{\displaystyle  \sum_{i=1}^{10^8}\mathds{1}_{\bm X_i \in \mathcal{L}_{MHD}(\alpha)}}.$$ 
Recall that Theorem \ref{MHD_vitesseCCTE} illustrates convergence rates in probability, however, for the sake of computational simplicity, we provide $\mathbb{L}^1$-estimation for the $\CCTE_D$ (which implies convergence results in probability). More precisely, we denote $\overline{\widehat{\CCTE}_{\alpha}^n}:=\overline{\widehat{\CCTE}_{\alpha, MHD}^n}$ the mean of the $\widehat{\CCTE}_{MHD,\alpha}^n$ based on $400$ simulations. The empirical standard deviation is
$$\hat{\sigma} = \displaystyle \sqrt{\frac{1}{399} \sum_{j=1}^{400} \left(\widehat{\CCTE}_{\alpha, j}^n- \overline{\widehat{\CCTE}_{\alpha}^n}\right)^2},$$
while the relative mean absolute error associated to $\widehat{\CCTE_{\alpha}^n}$, denoted by RMAE, is defined as follows : $$ \RMAE :={\RMAE}_{n,\alpha} = \displaystyle \frac{1}{400} \sum_{j=1}^{400} \frac{\big|\widehat{\CCTE}_{\alpha, j}^n-\CCTE_{MHD,\alpha}(Y,\bm X))\big|}{|\CCTE_{MHD,\alpha}(Y, \bm X)|}. $$
Note that, most of the times, one uses the Relative Mean Squared Error (RMSE) rather than the RMAE. However, since our results are presented with absolute value (particularly Theorem \ref{MHD_vitesseCCTE}), we work here with the RMAE for which we provide $\mathbb{L}^1$-estimation as well.
%$$ \RMSE := \displaystyle \sqrt{ \frac{1}{400} \sum_{j=1}^{400} \frac{\big|\widehat{\CCTE}_{\alpha, j}^n-\CCTE_\alpha(Y,\bm X))\big|^2}{|\CCTE_\alpha(Y, \bm X)|^2}}. $$
%où $\widehat{\CCTE}_{\alpha, j}^n$ est une $j$-ème estimation de la $\CCTE$ basée sur un échantillon de taille $n$, c'est-à-dire correspondant à la quantité : 
% $$ \widehat{\CCTE}_{\alpha,j}^n := \left(\displaystyle \frac{ \displaystyle \sum_{i=1}^{n}Y_i \mathds{1}_{\bm X_i \in \mathcal{L}_{n}(\alpha)}}{\displaystyle  \sum_{i=1}^{n}\mathds{1}_{\bm X_i \in \mathcal{L}_{n}(\alpha)}} \right)_j,$$
\\

In Table \ref{CCTE_RMAE_MHD}, we provide the above estimations for different values of $\alpha$ and sample size $n$. 
\begin{table}[H]\centering
\resizebox{12.2cm}{4cm}{
 \begin{tabular}{|c|lccccc|}
\hline 
 {\large\strut}  {$n$ } &  & $\alpha=0.1$ & $\alpha=0.2$ &$\alpha=0.5$ &$\alpha=0.8$ &$\alpha=0.9$ \tabularnewline
 {\large\strut}  & &   $\CCTE=1.4237$ & $\CCTE=0.8831$ &  $\CCTE=0.4804$ &$\CCTE=0.3831$ & $\CCTE=0.3661$ \tabularnewline
\hline\hline
\multirow{3}{0.8cm}{100}
 {\large\strut}  & {Mean } &   1.2744 &	0.8568 &	0.4706 & 0.3806 & 0.3636
  \\ 
 {\large\strut}  & {$\widehat{\sigma}$ }&0.6596 & 0.2678 & 0.0716 & 0.045 & 0.042 \tabularnewline
{\large\strut}  & {RMAE} & 0.3733 &	0.2360 &	 0.1217 & 0.0954 & 0.0933
  \\  
  %{\large\strut}  & {Rate} &  &	 &	 & & 
  %\\  
  \hline \hline
\multirow{3}{0.8cm}{1000}
 {\large\strut}  & {Mean } & 1.4125	& 0.8758 &0.4801	& 0.3835 & 0.3664 \\ % \cline{2-7}
 {\large\strut}  & {$\widehat{\sigma}$ } & 0.1923 &0.0782 & 0.0233 & 0.0153 & 0.0143\tabularnewline
 {\large\strut}  & {RMAE} &   0.1056 &	0.0696 &	0.0381 &	0.0318 & 0.0310
  \\  
%  {\large\strut}  & {Rate} &  0.533 &	0.058 &	0.014 &	0.01 & 0.01
 % \\  
\hline\hline
\multirow{3}{0.8cm}{5000}
 {\large\strut}  &{Mean} &  1.4213 & 	0.8808	 & 0.4804	& 0.3832& 0.3661 \tabularnewline %\cline{2-7}
 {\large\strut}  & {$\widehat{\sigma}$ } & 0.0855 & 0.0377 & 0.0106 & 0.0065 & 0.0061 \tabularnewline
 {\large\strut}  & {RMAE} &  0.0476 &	0.0338 & 0.018 &	0.0138 & 0.0134
  \\  
  %{\large\strut}  & {Rate} &   0.376 &	0.04 &	0.01 &	0.007 & 0.007
  %\\  
\hline \hline
\multirow{3}{0.8cm}{10000}
 {\large\strut}  & {Mean} & 1.4178	& 0.8807 & 0.4799	& 0.3828 & 0.3659 \\  %\cline{2-7}
 {\large\strut}  & {$\widehat{\sigma}$ } &0.0625 &0.0271 & 0.0071 & 0.0046&0.0043\tabularnewline
 {\large\strut}  & {RMAE} &  0.0361 &	0.0241 &	0.0117 &	0.0095 &  0.0094
  \\  
 % {\large\strut}  & {Rate} &  0.275 &	0.028 &	0.007 &	0.005 & 0.005
  %\\  
\hline \hline
\multirow{3}{0.7cm}{50000}
 {\large\strut}  & {Mean} & 1.4231&0.8829&0.4807&0.3833&0.3663\\  %\cline{2-7}
 {\large\strut}  & {$\widehat{\sigma}$ } &0.0283 &0.0118&0.0033&0.002&0.0019\tabularnewline
 {\large\strut}  & {RMAE} &0.0158 &0.0106& 0.0055&0.0042&0.0041
  \\  
 % {\large\strut}  & {Rate} &  0.275 &	0.028 &	0.007 &	0.005 & 0.005
  %\\  
\hline
\end{tabular}
}
\caption{$\mathbb{L}^1$-estimation of $\CCTE_{\alpha,MHD} (\mathbf{X},Y)$ and associated $\RMAE$ for bivariate Frank Copulas with Gumbel marginals.}\label{CCTE_RMAE_MHD} 
\end{table}
According to the results in Table \ref{CCTE_RMAE_MHD}, we observe that the error $\RMAE_n$ decreases as the sample size $n$ increases, as one may expect. Besides, remark that for low levels $\alpha$ ($\alpha = 0.1$) and sample sizes $n$ ($n=100,1000$), the value of $\RMAE$ is relatively high. This may be explained by the fact that, for small values of $\alpha$, there is fewer data to observe so that it becomes more difficult to estimate the mean $\widehat{\CCTE}$ as well as the $\alpha$-level set. Indeed, for low levels $\alpha$, the constant $A=2/m_{\nabla}$ is large since $m_{\nabla}$ approaches zero (see the proof of Proposition \ref{PropCasalRodriguez} and Theorem \ref{theoVitessedH} in Section \ref{proofs section}), meaning that the constant bounding the error ($\RMAE$) becomes large.  \\

In Table \ref{CCTEratesMHD}, we also provide $\mathbb{L}^1$-convergence rates : $V_{n,\alpha,\delta} = v_n(\delta) \cdot \RMAE_{n,\alpha},$ with  $v_n(\delta) = n^{\frac{1}{2}-\delta}$ for different values of $\delta$ (see Figure \ref{graph_vitesseMHD}).

\begin{table}[H]
\centering
 \begin{tabular}{|c|c||lcccc|}
\hline 
$\alpha$ & \diagbox{$\delta$}{$n$} & 100 & 1000 & 5000 &10000 &50000 \\
\hline
\multirow{4}{0.5cm}{0.1}
 & 0.05 &2.9656 & 2.3641 & 2.1971 & 2.2773 & 2.0529  \\ 
  & 0.01 & 3.5654 & 3.1164  & 3.0889 & 3.2916  & 3.1646\\
 & 0 & 3.7335 & 3.3393 & 3.3635& 3.6092 &  3.5262\\
 & -0.01 & 3.9094 &  3.5782& 3.6625& 3.9574 & 3.9292 \\
\hline \hline 
\multirow{4}{0.5cm}{0.2}
 & 0.05 & 1.8749 & 1.5589& 1.5602&1.5217 &    1.3821 \\ 
 & 0.01 & 2.2542& 2.0551 & 2.1935 & 2.1996 &  2.1305 \\
& 0 &  2.3604& 2.2021 & 2.3885 & 2.4118& 2.374 \\
& -0.01 & 2.4716 & 2.3595 & 2.6009  & 2.6445 &2.6452\\
\hline \hline 
\multirow{4}{0.5cm}{0.5}
 & 0.05 & 0.9666 & 0.8522 & 0.8298& 0.7407 &  0.7124 \\ 
& 0.01 & 1.1411 & 1.1234& 1.1666 &1.0707 & 1.0982\\
& 0 & 1.2169 & 1.2038& 1.2703& 1.1740& 1.2236 
 \\
& -0.01 &  1.2742 & 1.2899& 1.3832 &1.2872 & 1.3635 
\\
\hline \hline 
\multirow{4}{0.5cm}{0.8}
 & 0.05 & 0.7577   & 0.7119  &  0.6362 & 0.5989 &  0.5481 \\ 
 & 0.01 & 0.9111 &  0.9385 & 0.8945   & 0.8656& 0.845\\
& 0 & 0.9539 & 1.0056 & 0.9740 & 0.9491& 0.9415  \\
& -0.01 & 0.9988 &  1.0775  & 1.0606 & 1.0407 &1.0491 \\
\hline \hline
\multirow{4}{0.5cm}{0.9}
 & 0.05 & 0.7415 & 0.6949 & 0.6200 & 0.5952 & 0.5281  \\ 
 & 0.01 & 0.8914& 0.9160 & 0.8717 & 0.8603 & 0.8141\\
& 0 & 0.9334 & 0.9815 & 0.9492 & 0.9433&0.9072 \\
& -0.01 & 0.9774 & 1.0517 &  1.0335 & 1.0344 & 1.0108\\
\hline
\end{tabular}
\caption{Estimated $n^{\frac{1}{2}-\delta} \cdot \RMAE_{n,\alpha}$ based on $MHD$ for bivariate Frank Copulas with Gumbel marginals.} \label{CCTEratesMHD} 
\end{table}
\begin{figure}[H]
    \centering
    \includegraphics[scale=0.3]{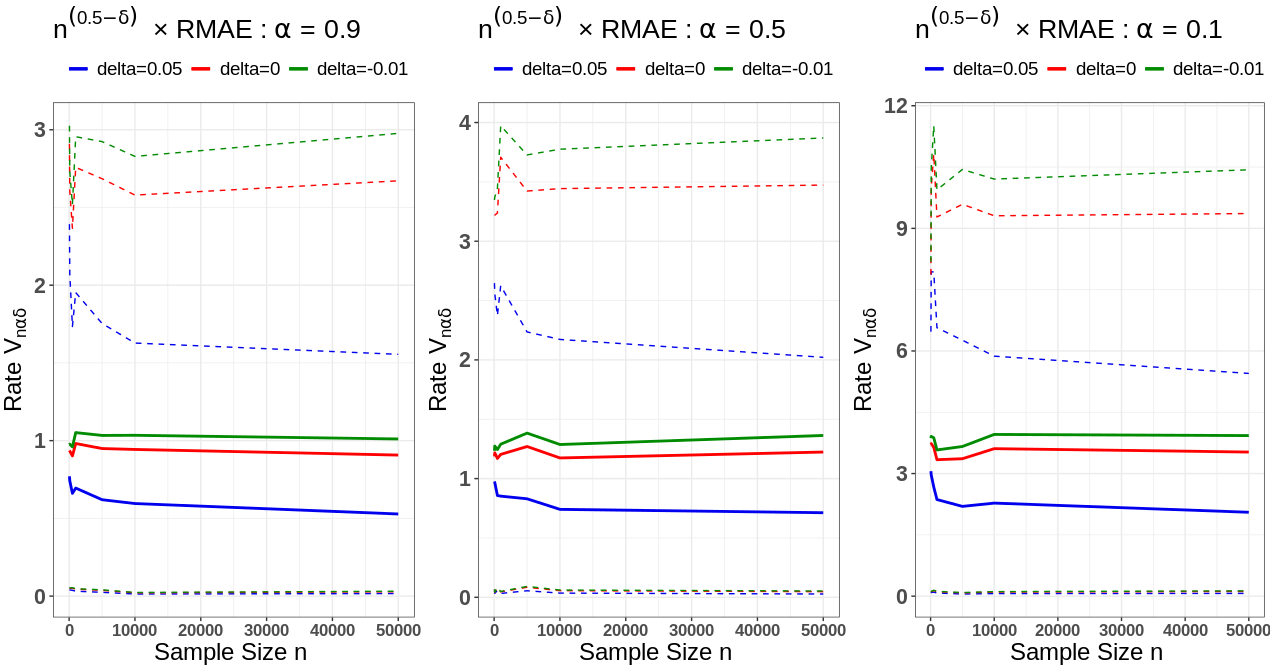}
    \caption{Estimation of the convergence rates $V_{n,\alpha,\delta}$ associated to RMAE, based on $MHD$ for bivariate Frank Copulas with Gumbel marginals.}
    \label{graph_vitesseMHD}
\end{figure}
In Table \ref{CCTEratesMHD} and Figure \ref{graph_vitesseMHD}, the parameter $\delta$ allows us to check that the critical regime is $\delta =0$ which corresponds to the $n^{1/2}$-convergence rate associated to the $\RMAE$. 
%\begin{figure}[H]
  %  \centering
  % \includegraphics[scale=0.35]{rate_graph.png}
%\end{figure}
%\subsection{A real data study} \label{Section real data}
\section{Proofs of Section \ref{mainresults}} \label{proofs section}
\textbf{Proofs of Section \ref{general CCTE section}.}\\
%In what follows, the results are proved in the sense of convergence in probability, which directly implies the $\mathcal{O}_P$ bounds.
\begin{proof}[Proof of Theorem \ref{vitesseCCTEproba}]
We can write \begin{align*}
    \displaystyle  &\big|  \widehat{\CCTE}_{D,\alpha}^{n_1,n_2}(Y,\bm X) - {\CCTE}
    _{D,\alpha}(Y,\bm X)  \big| \cdot \mathds{1}_{P_{\tilde{S}_{n_1}}[\mathcal{L}_{n_1}(\alpha)]>0}\\
    &= \displaystyle \left| \frac{\frac{1}{n_2} \sum_{i=1}^{n_2} Y_i \mathds{1}_{\bm X_i\in \mathcal{L}_{n_1}(\alpha)}}{\frac{1}{n_2}\sum_{i=1}^{n_2} \mathds{1}_{\bm X_i\in \mathcal{L}_{n_1}(\alpha)}}- \mathbb{E}[Y|\bm X \in \mathcal{L}_D(\alpha)]   \right| \cdot \mathds{1}_{P_{\tilde{S}_{n_1}}[\mathcal{L}_{n_1}(\alpha)]>0} \\
    & \leq \displaystyle \left|  \frac{\frac{1}{n_2} \sum_{i=1}^{n_2} Y_i \mathds{1}_{\bm X_i\in \mathcal{L}_{n_1}(\alpha)}}{\frac{1}{n_2}\sum_{i=1}^{n_2} \mathds{1}_{\bm X_i\in \mathcal{L}_{n_1}(\alpha)}} - \mathbb{E}_{\tilde{S}_{n_1}}[Y|\bm X \in \mathcal{L}_{n_1}(\alpha)]   \right|\cdot \mathds{1}_{P_{\tilde{S}_{n_1}}[\mathcal{L}_{n_1}(\alpha)]>0} \\
     &~~+ \displaystyle \left| \mathbb{E}_{\tilde{S}_{n_1}}[Y|\bm X \in \mathcal{L}_{n_1}(\alpha)] - \mathbb{E}[Y|\bm X \in \mathcal{L}_D(\alpha)]   \right|\cdot \mathds{1}_{P_{\tilde{S}_{n_1}}[\mathcal{L}_{n_1}(\alpha)]>0}.
\end{align*}

The proof of Theorem \ref{vitesseCCTEproba} is a modified version of the proof of Theorem 5.1 in \citet{DBLS2015}. The latter focuses on distribution functions instead of depth functions. Besides, in the proof of Theorem \ref{vitesseCCTEproba}, we show that $\mathds{1}_{P_{\tilde{S}_{n_1}}[\mathcal{L}_{n_1}(\alpha)]>0}$ converges to one in probability.\\

First recall that, in the following, probability measures involving events which depend on $\mathcal{L}_{n_1}(\alpha)$ are conditional expectations to the sample $\tilde{S}_{n_1}$. For notational convenience, we recall that $$P_{\tilde{S}_{n_1}}[\mathcal{L}_{n_1}(\alpha)]:=\mathbb{P}[\bm X \in \mathcal{L}_{n_1}(\alpha)]$$ which is a random variable. Moreover, note that the convergence to zero in probability implies directly the $\mathcal{O}_P(1)$ result.\\

The proof of Theorem \ref{vitesseCCTEproba} is based on the two following preliminary results, Lemma \ref{lemmeccte1} and Lemma \ref{lemmeccte2}. 
\end{proof}
\begin{lemma}\label{lemmeccte1}
Under assumptions of Theorem \ref{vitesseCCTEproba}, it holds that
\begin{equation*}
 \left| \mathbb{E}_{\tilde{S}_{n_1}}[Y|\bm X \in \mathcal{L}_{n_1}(\alpha)] - \mathbb{E}[Y|\bm X \in \mathcal{L}_D(\alpha)]   \right| = \mathcal{O}_{P,n_1}\left( v_{n_1}^{-\left(1-\frac{1}{r}\right)}\right).
\end{equation*}
\end{lemma}
\begin{proof}[Proof of Lemma \ref{lemmeccte1}]
On the event $\left\lbrace P_{\tilde{S}_{n_1}}[\mathcal{L}_{n_1}(\alpha)]>0 \right\rbrace$, it holds
\begin{multline*}
    v_{n_1}^{1-\frac{1}{r}} \displaystyle \left| \mathbb{E}_{\tilde{S}_{n_1}}[Y|\bm X \in \mathcal{L}_{n_1}(\alpha)] - \mathbb{E}[Y|\bm X \in \mathcal{L}_D(\alpha)]   \right| \\
    =v_{n_1}^{1-\frac{1}{r}} \left| \displaystyle \frac{ \mathbb{E}_{\tilde{S}_{n_1}}[Y\mathds{1}_{\bm X \in \mathcal{L}_{n_1}(\alpha)}]}{P_{\tilde{S}_{n_1}}[\mathcal{L}_{n_1}(\alpha)]} -  \frac{ \mathbb{E}[Y\mathds{1}_{\bm X \in \mathcal{L}_D(\alpha)}]}{P[\mathcal{L}_D(\alpha)]} \right| \hspace{5cm}\\
      \leq \displaystyle \frac{v_{n_1}^{1-\frac{1}{r}}}{P_{\tilde{S}_{n_1}}[\mathcal{L}_{n_1}(\alpha)]P[ \mathcal{L}_D(\alpha)]} \times \bigg( \hspace{7.5cm} \\ P[\mathcal{L}_D(\alpha)]  \Big| \mathbb{E}_{\tilde{S}_{n_1}}[Y\mathds{1}_{\bm X \in \mathcal{L}_{n_1}(\alpha)}]
        - \mathbb{E}[Y\mathds{1}_{\bm X \in \mathcal{L}_D(\alpha)}] \Big| \\
       + \displaystyle \left| P[\mathcal{L}_D(\alpha)]- P_{\tilde{S}_{n_1}}[\mathcal{L}_{n_1}(\alpha)] \right| \left|\mathbb{E}[Y\mathds{1}_{\bm X \in \mathcal{L}_D(\alpha)}]\right|
      \bigg) \\
      \leq \displaystyle \frac{v_{n_1}^{1-\frac{1}{r}}}{P_{\tilde{S}_{n_1}}[\mathcal{L}_{n_1}(\alpha)]P[\mathcal{L}_D(\alpha)]} \times \bigg( \Big| \mathbb{E}_{\tilde{S}_{n_1}}[Y \mathds{1}_{\bm X \in \mathcal{L}_{n_1}(\alpha)}]  - \mathbb{E}[Y \mathds{1}_{\bm X \in \mathcal{L}_D(\alpha)}]  \Big| \hspace{1cm}\\
       + \displaystyle \mathbb{E}[|Y|]\Big| P[\mathcal{L}_D(\alpha)]- P_{\tilde{S}_{n_1}}[\mathcal{L}_{n_1}(\alpha)] \Big| \bigg).\hspace{1cm}
\end{multline*}
Under Assumption \textbf{(H0)} and since $v_{n_1}^{-1} \to 0$ as $n_1 \to \infty$, it holds that $$P_{\tilde{S}_{n_1}}[\mathcal{L}_{n_1}(\alpha)\Delta \mathcal{L}_D(\alpha)] \xrightarrow[n_1 \to \infty]{\mathbb{P}}0,$$ so that
 $$P_{\tilde{S}_{n_1}}[ \mathcal{L}_{n_1}(\alpha)] \xrightarrow[n_1 \to \infty]{\mathbb{P}} P[\mathcal{L}_D(\alpha)]>0,$$
 and \begin{equation} \label{proba1event}
     \mathbb{P}\left[\left\lbrace P_{\tilde{S}_{n_1}}[\mathcal{L}_{n_1}(\alpha)]>0 \right\rbrace \right] \xrightarrow[n_1 \to \infty]{} 1.
 \end{equation} 
On the one hand, 
\begin{align*}
   v_{n_1}^{1-\frac{1}{r}}\left|P_{\tilde{S}_{n_1}}[\mathcal{L}_{n_1}(\alpha)]-P[\mathcal{L}_D(\alpha)] \right|
    \leq  v_{n_1}^{1-\frac{1}{r}} P_{\tilde{S}_{n_1}}[ \mathcal{L}_D(\alpha) \Delta \mathcal{L}_{n_1}(\alpha)],
\end{align*}
so we obtain\begin{equation} \label{vi1}
    v_{n_1}^{1-\frac{1}{r}} \left|P_{\tilde{S}_{n_1}}[\mathcal{L}_{n_1}(\alpha)]-P[\mathcal{L}_D(\alpha)] \right| \xrightarrow[n_1 \to \infty]{\mathbb{P}} 0.
\end{equation}
On the other hand, using Hölder inequality
\begin{align*}
   v_{n_1}^{1-\frac{1}{r}} \big| \mathbb{E}_{\tilde{S}_{n_1}}[Y \mathds{1}_{\bm X \in \mathcal{L}_n(\alpha)}]&  - \mathbb{E}[Y \mathds{1}_{\bm X \in \mathcal{L}_D(\alpha)}]  \big| \\ \notag
    \leq & v_{n_1}^{1-\frac{1}{r}} \mathbb{E}_{\tilde{S}_{n_1}}\big[|Y| \mathds{1}_{\bm X \in \mathcal{L}_{n_1}(\alpha)\Delta \mathcal{L}_D(\alpha)}\big]  \\ \notag
    \leq& v_{n_1}^{1-\frac{1}{r}} \mathbb{E}\big[|Y|^r \big]^{\frac{1}{r}}\mathbb{E}_{\tilde{S}_{n_1}}\big[\mathds{1}_{\bm X \in \mathcal{L}_{n_1}(\alpha)\Delta \mathcal{L}_D(\alpha)}\big]^{1-\frac{1}{r}}  \\ \notag \\
    = & v_{n_1}^{1-\frac{1}{r}}P_{\tilde{S}_{n_1}}[\mathcal{L}_{n_1}(\alpha) \Delta \mathcal{L}_D(\alpha)]^{1-\frac{1}{r}} \|Y\|_{\mathbb{L}^r(\Omega)}  .
\end{align*}
Since $v_{n_1}P_{\tilde{S}_{n_1}}[\mathcal{L}_{n_1}(\alpha) \Delta \mathcal{L}_D(\alpha)]=\mathcal{O}_{P,n_1}(1)$, it holds
\begin{equation} \label{vi2}
   v_{n_1}^{1-\frac{1}{r}} \big| \mathbb{E}_{\tilde{S}_{n_1}}[Y \mathds{1}_{\mathcal{L}_{n_1}(\alpha)}]  - \mathbb{E}[Y \mathds{1}_{\mathcal{L}_D(\alpha)}]  \big| = \mathcal{O}_{P,n_1}(1).
\end{equation}
Since the convergence to zero in probability implies the $\mathcal{O}_P(1)$ result, the lemma follows directly from (\ref{proba1event}, (\ref{vi1}), and (\ref{vi2}).
The case $r=+\infty$ is analogous.\\
\end{proof}
\begin{lemma}\label{lemmeccte2}
Under assumptions of Theorem \ref{vitesseCCTEproba}, we obtain \begin{equation*}
    \displaystyle \left|  \frac{\frac{1}{n_2} \sum_{i=1}^{n_2} Y_i \mathds{1}_{\bm X_i\in \mathcal{L}_{n_1}(\alpha)}}{\frac{1}{n_2}\sum_{i=1}^{n_2} \mathds{1}_{\bm X_i\in \mathcal{L}_{n_1}(\alpha)}} - \mathbb{E}[Y|\bm X \in \mathcal{L}_{n_1}(\alpha)]   \right|= \mathcal{O}_{P,n_1,n_2}(n_2^{-\frac{1}{2}}).
\end{equation*} 
\end{lemma}
\begin{proof}[Proof of Lemma \ref{lemmeccte2}] First, we distinguish the event in which \\ $\frac{1}{n_2}\sum_{i=1}^{n_2} \mathds{1}_{\bm X_i\in \mathcal{L}_{n_1}(\alpha)} =0$, then the one in which $\frac{1}{n_2}\sum_{i=1}^{n_2} \mathds{1}_{\bm X_i\in \mathcal{L}_{n_1}(\alpha)} \neq 0.$ \newline
Let $0< \varepsilon <P[\mathcal{L}_D(\alpha)]. $
Since the $(Y_i,\bm X_i)_{1 \leq i \leq n_2}$ are iid so that the $(\bm X_i)_{1 \leq i \leq n_2}$ are iid, it holds that
\begin{align*}
    \mathbb{P}& \left[\frac{1}{n_2}\sum_{i=1}^{n_2} \mathds{1}_{\bm X_i\in \mathcal{L}_{n_1}(\alpha)} = 0 \right] \\
& =  
\mathbb{E}\left[P_{\tilde{S}_{n_1}}\left[\frac{1}{n_2}\sum_{i=1}^{n_2} \mathds{1}_{\bm X_i\in \mathcal{L}_{n_1}(\alpha)} =0 \right]\right] \\
& = \mathbb{E}\left[ \prod_{i=1}^{n_2}P_{\tilde{S}_{n_1}}\left[ \bm X_i \notin \mathcal{L}_{n_1}(\alpha) \right]\right] \\
& = \mathbb{E}\left[ P_{\tilde{S}_{n_1}}\left[ \bm X \notin \mathcal{L}_{n_1}(\alpha) \right]^{n_2}\right] \\
%& =\mathbb{E}\left[ (1-P_{\tilde{S}_{n_1}}[\mathcal{L}_{n_1}(\alpha)])^{n_2}\right] \\
& = \mathbb{E}\left[(1-P_{\tilde{S}_{n_1}}[\mathcal{L}_{n_1}(\alpha)])^{n_2} \mathds{1}_{P_{\tilde{S}_{n_1}}[\mathcal{L}_{n_1}(\alpha)] \geq \varepsilon}\right] \\
& \hspace{4cm} + \mathbb{E}\left[(1-P_{\tilde{S}_{n_1}}[\mathcal{L}_{n_1}(\alpha)])^{n_2} \mathds{1}_{P_{\tilde{S}_{n_1}}[\mathcal{L}_{n_1}(\alpha)] < \varepsilon}\right]  \\
& \leq (1-\varepsilon)^{n_2} + \mathbb{P}\left[P_{\tilde{S}_{n_1}}[\mathcal{L}_{n_1}(\alpha)] < \varepsilon \right].
\end{align*}
Since $\varepsilon \in \left(0, P[\mathcal{L}_D(\alpha)] \right)$ and $P_{\tilde{S}_{n_1}}[\mathcal{L}_{n_1}(\alpha)] \xrightarrow[n_1 \to \infty]{\mathbb{P}} P[\mathcal{L}_D(\alpha)]$ (see Lemma \ref{lemmeccte1}), we obtain $\mathbb{P}  \left[\frac{1}{n_2}\sum_{i=1}^{n_2} \mathds{1}_{\bm X_i\in \mathcal{L}_{n_1}(\alpha)} = 0 \right] \xrightarrow[n_1,n_2 \to \infty]{}0.$ Now on the event \\ $\left\lbrace \frac{1}{n_2}\sum_{i=1}^{n_2} \mathds{1}_{\bm X_i\in \mathcal{L}_{n_1}(\alpha)} \neq 0 \right\rbrace \cap \left\lbrace P_{\tilde{S}_{n_1}}[\mathcal{L}_{n_1}(\alpha)] >0 \right\rbrace$,
we can write
\begin{align*}\label{ineq1}
    \Bigg| & \frac{\frac{1}{n_2} \sum_{i=1}^{n_2} Y_i \mathds{1}_{\bm X_i\in \mathcal{L}_{n_1}(\alpha)}}{\frac{1}{n_2}\sum_{i=1}^{n_2} \mathds{1}_{\bm X_i\in \mathcal{L}_{n_1}(\alpha)}}- \mathbb{E}_{\tilde{S}_{n_1}}[Y|\bm X \in \mathcal{L}_{n_1}(\alpha)]  \Bigg|  \\
    = & \left| \frac{ \sum_{i=1}^{n_2} Y_i \mathds{1}_{\bm X_i\in \mathcal{L}_{n_1}(\alpha)}}{\sum_{i=1}^{n_2} \mathds{1}_{\bm X_i\in \mathcal{L}_{n_1}(\alpha)}}
    - \frac{\mathbb{E}_{\tilde{S}_{n_1}}[Y \mathds{1}_{\bm X \in \mathcal{L}_{n_1}(\alpha)}]}{P_{\tilde{S}_{n_1}}[\mathcal{L}_{n_1}(\alpha)]}   \right| \\
    \leq & \frac{1}{\frac{1}{n_2}\sum_{i=1}^{n_2} \mathds{1}_{\bm X_i \in \mathcal{L}_{n_1}(\alpha)}}\bigg| \frac{1}{n_2} \sum_{i=1}^{n_2} Y_i \mathds{1}_{\bm X_i\in \mathcal{L}_{n_1}(\alpha)} - \mathbb{E}_{\tilde{S}_{n_1}}[Y \mathds{1}_{\bm X \in \mathcal{L}_{n_1}(\alpha)}]  \bigg| \\
    & +\big| \mathbb{E}_{\tilde{S}_{n_1}}[Y \mathds{1}_{\bm X \in \mathcal{L}_{n_1}(\alpha)}] \big| \bigg| \frac{1}{\frac{1}{n_2}\sum_{i=1}^{n_2}\mathds{1}_{\bm X_i\in \mathcal{L}_{n_1}(\alpha)}}-\frac{1}{P_{\tilde{S}_{n_1}}[\mathcal{L}_{n_1}(\alpha)]} \bigg|\\
    \leq & \frac{1}{\frac{1}{n_2}\sum_{i=1}^{n_2} \mathds{1}_{\bm X_i \in \mathcal{L}_{n_1}(\alpha)}}\bigg| \frac{1}{n_2} \sum_{i=1}^{n_2} Y_i \mathds{1}_{\bm X_i\in \mathcal{L}_{n_1}(\alpha)} -\mathbb{E}_{\tilde{S}_{n_1}}[Y \mathds{1}_{\bm X \in \mathcal{L}_{n_1}(\alpha)}] \bigg| \\
    & + \frac{\mathbb{E}[|Y|]}{P_{\tilde{S}_{n_1}}[\mathcal{L}_{n_1}(\alpha)]\frac{1}{n_2}\sum_{i=1}^{n_2}\mathds{1}_{\bm X_i\in \mathcal{L}_{n_1}(\alpha)}} \bigg|\frac{1}{n_2}\sum_{i=1}^{n_2}\mathds{1}_{\bm X_i\in \mathcal{L}_{n_1}(\alpha)}-P_{\tilde{S}_{n_1}}[\mathcal{L}_{n_1}(\alpha)] \bigg|. \tag{\textbf{R}}
\end{align*}
Let us first clarify the convergence of the denominator terms. Recall that under Assumption \textbf{(H0)} of Theorem \ref{vitesseCCTEproba}, $$ P_{\tilde{S}_{n_1}}[ \mathcal{L}_{n_1}(\alpha)] \xrightarrow[n_1 \to \infty]{\mathbb{P}} P[\mathcal{L}_D(\alpha)]>0,$$
and $\mathbb{P}\left[ \left\lbrace P_{\tilde{S}_{n_1}}[\mathcal{L}_{n_1}(\alpha)]>0 \right\rbrace \right] \xrightarrow[n_1 \to \infty]{} 1$ (see the proof of Lemma \ref{lemmeccte1}).
Next, we prove %for all $\delta>0$ 
\begin{equation}\label{vittermeindi}
    n_2^{\frac{1}{2}} \left( \displaystyle \frac{1}{n_2}\sum_{i=1}^{n_2}\mathds{1}_{\bm X_i\in \mathcal{L}_{n_1}(\alpha)}-P_{\tilde{S}_{n_1}}[\mathcal{L}_{n_1}(\alpha)] \right) =\mathcal{O}_{P,n_1,n_2}(1),
\end{equation}
so that we obtain \begin{equation}
     \displaystyle \frac{1}{n_2}\sum_{i=1}^{n_2}\mathds{1}_{\bm X_i\in \mathcal{L}_{n_1}(\alpha)}-P_{\tilde{S}_{n_1}}[\mathcal{L}_{n_1}(\alpha)] \xrightarrow[\substack{n_1, n_2 \to \infty}]{\mathbb{P}} 0,
\end{equation} 
and $$  \frac{1}{n_2}\sum_{i=1}^{n_2}\mathds{1}_{\bm X_i\in \mathcal{L}_{n_1}(\alpha)}  \xrightarrow[\substack{n_1, n_2 \to \infty}]{\mathbb{P}} P[\mathcal{L}_D(\alpha)]>0.$$
Let us prove (\ref{vittermeindi}). Let $n_1 \geq 1$ and $\varepsilon>0$. Using Tchebychev inequality, we can write $\mathbb{P}$-a.s (here the event $\omega \in \Omega$ is one realisation of the sample $\tilde{S}_{n_1}$ and is independent of $\varepsilon$) \begin{align*}
    P_{\tilde{S}_{n_1}}\bigg[ \Big| \displaystyle \frac{1}{n_2}\sum_{i=1}^{n_2}\mathds{1}_{\bm X_i\in \mathcal{L}_{n_1}(\alpha)}-P_{\tilde{S}_{n_1}}[\mathcal{L}_{n_1}(\alpha)] \Big| \geq \varepsilon \bigg] & \leq \displaystyle \frac{\mathbb{V}_{\tilde{S}_{n_1}}\big(\frac{1}{n_2}\sum_{i=1}^{n_2}\mathds{1}_{\bm X_i\in \mathcal{L}_{n_1}(\alpha)}\big)}{\varepsilon^2} \\
    & = \displaystyle \frac{\mathbb{V}_{\tilde{S}_{n_1}} \big(\mathds{1}_{\bm X_1\in \mathcal{L}_{n_1}(\alpha)}\big)}{n_2 \varepsilon^2} \\
    & \leq \frac{1}{n_2\varepsilon^2}.
\end{align*}
Thus, taking $M_\varepsilon:=1/\varepsilon^{\frac{1}{2}}$ it holds that
\begin{align*}
    \sup_{n_1,n_2 \geq 1}\mathbb{P} \bigg[ n_2^{\frac{1}{2}} \Big| \displaystyle \frac{1}{n_2}\sum_{i=1}^{n_2}\mathds{1}_{\bm X_i\in \mathcal{L}_{n_1}(\alpha)}-P_{\tilde{S}_{n_1}}[\mathcal{L}_{n_1}(\alpha)] \Big| \geq M_\varepsilon \bigg] & \leq \frac{1}{n_2\left(\frac{M_\varepsilon}{n_2^{\frac{1}{2}}}\right)^2} \\
    &=\varepsilon,
\end{align*}
which means that (\ref{vittermeindi}) is satisfied.
Similarly, we obtain
$$ \frac{1}{n_2} \sum_{i=1}^{n_2} Y_i \mathds{1}_{\bm X_i\in \mathcal{L}_{n_1}(\alpha)} -\mathbb{E}_{\tilde{S}_{n_1}}[Y \mathds{1}_{\bm X \in \mathcal{L}_{n_1}(\alpha)}] = \mathcal{O}_{P,n_1,n_2}\left(n_2^{\frac{1}{2}}\right)$$
%\begin{align*}
%    \mathbb{P} \bigg[  n_2^{\frac{1}{2}-\delta}\bigg| \frac{1}{n_2} \sum_{i=1}^{n_2} Y_i \mathds{1}_{\bm X_i\in \mathcal{L}_{n_1}(\alpha)} -\mathbb{E}[Y \mathds{1}_{\bm X \in \mathcal{L}_{n_1}(\alpha)}] \bigg| \geq \varepsilon \bigg] & \leq \displaystyle \mathbb{E}[Y^2] \frac{n_2^{-2\delta}}{\varepsilon^2} \\
 %   & \xrightarrow[n_1,n_2 \to \infty]{}0.
%\end{align*}
Hence the result.
\end{proof}
\textbf{Proofs of Section \ref{Hausdorff section}.}
\begin{remark}\label{remarquedPropCasalRod}
Among the four properties of a depth function, the only necessary property in Proposition \ref{PropCasalRodriguez} is \textbf{(D4)} (vanishing at infinity). The latter guarantees that the set $\mathcal{K}_\varepsilon(\alpha)$ is compact in $\mathbb{R}^d.$ Indeed, as $D$ satisfies \textbf{(D4)}, the assumption $0<\varepsilon<\alpha$ implies $\mathcal{K}_\varepsilon(\alpha)$ is bounded, moreover under (i), $D$ is continuous on $\mathcal{K}_\varepsilon(\alpha)$ so that $\mathcal{K}_\varepsilon(\alpha)$ is a closed set. By denoting $(HD)_x$ the Hessian matrix of $D$ at x, one can note that $M_H:=\sup_{x \in \mathcal{K}_\varepsilon(\alpha)}\opnorm{(HD)_x}<\infty$, as a supremum of a continuous mapping on a compact set.
\end{remark}
\begin{proof}[Proof of Proposition \ref{PropCasalRodriguez}]
For the sake of simplicity, denote by $[\alpha \pm \varepsilon]$ the interval $[\alpha-\varepsilon, \alpha+ \varepsilon]$. Under the assumptions of Proposition \ref{PropCasalRodriguez}, $K_\varepsilon(\alpha):=D^{-1}([\alpha \pm \varepsilon])$ is compact and $$M_H:=\sup_{x \in K_\varepsilon(\alpha)}\opnorm{HD_x} < \infty$$ (cf. Remark \ref{remarquedPropCasalRod}).
We state the following useful lemma.
\begin{lemma}\label{lemmeboule}
Under the assumptions of Proposition \ref{PropCasalRodriguez}, there exist $N:=N_\varepsilon \geq 1$, some points $x_i:=x_{i,\varepsilon} \in \mathcal{K}_{\frac{\varepsilon}{2}}(\alpha)$, and some positive real numbers $r_i:=r_{x_i} \in \mathbb{R}_+^*$, $1 \leq i \leq N$, s.t. \begin{center}
     $\displaystyle \mathcal{K}_{\frac{\varepsilon}{2}}(\alpha) \subset \bigcup_{i=1}^N \Boule \bigg(x_i,\frac{r_i}{2} \bigg) \subset \bigcup_{i=1}^N \Boule(x_i,r_i) \subset  \mathcal{K}_{\varepsilon}(\alpha).$
 \end{center}
%In particular, for all $0<\gamma\leq \varepsilon/2,$
% \begin{center}
 %    $\displaystyle \mathcal{K}_{\gamma}(\alpha) \subset \bigcup_{i=1}^N \Boule \bigg(x_i,\frac{r_i}{2} \bigg) \subset \bigcup_{i=1}^N \Boule(x_i,r_i) \subset  \mathcal{K}_{\varepsilon}(\alpha).$
% \end{center}
\end{lemma}
\begin{proof}[Proof of Lemma \ref{lemmeboule}]
Since $\varepsilon >0$, then $K_{\frac{\varepsilon}{2}}(\alpha)$ is a subset of the interior of $\mathcal{K}_{\varepsilon}(\alpha)$. Thus, for any $x \in K_{\frac{\varepsilon}{2}}(\alpha)$, there exists $r_x:=r_x(\varepsilon)>0$ s.t. $$\Boule(x,r_x) \subset \mathcal{K}_\varepsilon(\alpha),$$ that is,
$$\displaystyle K_{\frac{\varepsilon}{2}}(\alpha) \subset \bigcup_{x \in K_{\frac{\varepsilon}{2}}(\alpha)} \Boule \bigg(x,\frac{r_x}{2} \bigg) \subset \bigcup_{x \in K_{\frac{\varepsilon}{2}}(\alpha)} \Boule(x,r_x) \subset \mathcal{K}_\varepsilon(\alpha).$$
Since we have an open cover of the compact set $\displaystyle K_{\frac{\varepsilon}{2}}(\alpha)$, then the latter has a finite cover. In other words, there exist $N \geq 1$, some points $x_i:=x_{i,\varepsilon} \in \mathcal{K}_{\frac{\varepsilon}{2}}(\alpha)$, and $r_i:=r_{x_i} \in \mathbb{R}_+^*$, $1 \leq i \leq N$, s.t.\begin{center}
     $\displaystyle \mathcal{K}_{\frac{\varepsilon}{2}}(\alpha) \subset \bigcup_{i=1}^N \Boule \bigg(x_i,\frac{r_i}{2} \bigg) \subset \bigcup_{i=1}^N \Boule(x_i,r_i)  \subset \mathcal{K}_\varepsilon(\alpha).$
 \end{center}
 Hence the result.
\end{proof}
Let $0<\gamma \leq \varepsilon/2$ and $x \in \mathcal{K}_\gamma(\alpha)$. For $\lambda \in \mathbb{R}$, define $$y_\lambda:=y_{\lambda,x}=x+\lambda\frac{(\nabla D)_x}{\|(\nabla D)_x\|},$$
with $\|(\nabla D)_x\| \geq m_\nabla >0,$ since $\mathcal{K}_\gamma(\alpha) \subset \mathcal{K}_\varepsilon(\alpha)$. In what follows, we take $$\| y_\lambda-x \|=|\lambda| < \displaystyle \min_{1 \leq i \leq N} \frac{r_i}{2}.$$ It holds $[y_\lambda,x] \subset \mathcal{K}_\varepsilon(\alpha)$. Indeed, $x \in \mathcal{K}_\gamma(\alpha)$ so that Lemma \ref{lemmeboule} applies, namely, there exists $ 1 \leq i_0 \leq N$ s.t. $x \in \Boule(x_{i_0},r_{i_0}/2),$ and for all $z \in [y_\lambda,x],$ \begin{align*}
    \|z-x_{i_0}\| &\leq \|z-x\|+\|x-x_{i_0}\| \\
    & \leq \|y_\lambda-x\|+\|x-x_{i_0}\| \\
    &=|\lambda|+\|x-x_{i_0}\| \\
    & < \min_{1 \leq i \leq N}\frac{r_i}{2}+ \frac{r_{i_0}}{2}\\
    & \leq r_{i_0}.
\end{align*}
Thus, $z \in \Boule(x_{i_0},r_{i_0}) \subset \mathcal{K}_\varepsilon(\alpha)$ (cf. Lemma \ref{lemmeboule}).
%The function $D(\cdot)$ is of class $\mathscr{C}^2$ on $\mathcal{K}_\varepsilon(\alpha)$, with the fact that
Since $\displaystyle |\lambda| < \min_{1 \leq i \leq N} r_i/2$, using a Taylor expansion on the line $[x,y_\lambda] \subset \mathcal{K}_\varepsilon(\alpha)$, it holds \begin{center}
     $\displaystyle D(y_\lambda)=D(x)+\langle (\nabla D)_x, y_\lambda-x\rangle + \frac{1}{2} \left\langle y_\lambda - x, (HD)_{\overline{x}}(y_\lambda-x) \right\rangle,$ $\overline{x} \in [x,y_\lambda],$
 \end{center} 
then, %using Cauchy-Schwarz inequality,
 \begin{align*}
      \displaystyle D(y_\lambda)=& D(x)+ \lambda \|(\nabla D)_x\| + \frac{\lambda^2}{2\|(\nabla D)_x\|^2}\left\langle (\nabla D)_x, (HD)_{\overline{x}}(\nabla D)_x \right\rangle.
 \end{align*}
Using Cauchy-Schwarz inequality, it holds
\begin{align*}
 |D(y_\lambda) -D(x) - \lambda \|(\nabla D)_x\| |\leq~&  \frac{\lambda^2}{2\|(\nabla D)_x\|^2}\| (\nabla D)_x \| \opnorm{(HD)_{\overline{x}}} \cdot \|(\nabla D)_x \| \\
     = & \frac{\lambda^2}{2}\opnorm{(HD)_{\overline{x}}}.
\end{align*}
 Since $\overline{x} \in \mathcal{K}_\varepsilon(\alpha)$, then $ \opnorm{(HD)_{\overline{x}}} \leq \sup_{x \in \mathcal{K}_\varepsilon(\alpha)} \opnorm{(HD)_{\overline{x}}} =M_H < \infty.$
For any $\displaystyle |\lambda| < \min_{1 \leq i \leq N} r_i/2$, we obtain
\begin{equation}\label{E.0}
    D(x) + \lambda \|(\nabla D)_x\| - \frac{\lambda^2}{2}M_H \leq D(y_\lambda) \leq D(x) + \lambda \|(\nabla D)_x\| + \frac{\lambda^2}{2}M_H.
\end{equation}
If $\displaystyle 0<\lambda<\min_{1 \leq i \leq N} r_i/2$, then with the above inequality, we have \begin{equation}\label{E.1}
  \displaystyle D(y_\lambda) \geq  D(x) + \lambda \inf_{x \in \mathcal{K}_\varepsilon(\alpha)} \|(\nabla D)_x\| - \frac{\lambda^2}{2}M_H =D(x)+\lambda \left(m_\nabla-\lambda\frac{M_H}{2}\right)
\end{equation}
Suppose now $M_H>0$ (the case $M_H=0$ is trivial). That way, if $ 0<\lambda<\frac{m_\nabla}{M_H} \wedge \min_{1 \leq i \leq N} r_i/2$, %(i.e. $\displaystyle 0<\lambda<\min_{1 \leq i \leq N} r_i/2$, and $\lambda<\frac{m_\nabla}{M_H}$ so that $m_\nabla-\lambda\frac{M_H}{2} \geq \frac{m_\nabla}{2}$)
using (\ref{E.1}), \begin{center}
   $ D(y_\lambda) \geq D(x)+\lambda \frac{m_\nabla}{2}$.
\end{center}
Similarly, using the right hand side of inequality (\ref{E.0}), for any $\displaystyle 0<\lambda < \frac{m_\nabla}{M_H} \wedge \min_{1 \leq i \leq N} \frac{r_i}{2} $, $$D(y_{-\lambda})\leq D(x)-\lambda\frac{m_\nabla}{2}.$$
To sum up, for any $0<\gamma \leq \varepsilon/2$, $x \in \mathcal{K}_\gamma(\alpha)$ and  $\displaystyle 0<\lambda<\frac{m_\nabla}{M_H} \wedge \min_{1 \leq i \leq N} \frac{r_i}{2}$, it holds \begin{align}
    &D(y_\lambda) \geq D(x) +\lambda\frac{m_\nabla}{2}, \label{1*} \\
    &D(y_{-\lambda}) \leq D(x) -\lambda\frac{m_\nabla}{2}. \label{2*}
\end{align}
Choose $\displaystyle \gamma:=\displaystyle \left[\frac{m_\nabla}{4}\left( \frac{m_\nabla}{M_H} \wedge \min_{1\leq i \leq N} \frac{r_i}{2}\right) \right]\wedge \frac{\varepsilon}{2}>0.$ Now we show :\begin{center}
if $|\alpha-\beta|  \leq \gamma $, then $d_H(\left\lbrace D=\alpha \right\rbrace,\left\lbrace D=\beta \right\rbrace) \leq \frac{2}{m_\nabla}|\alpha-\beta|$.
\end{center} 
Assume $|\alpha-\beta|  \leq \gamma $.\\
Let $\beta$ be s.t. $0<\beta-\alpha \leq \gamma$. In this case, $\beta=\alpha+\eta$ with $0<\eta \leq \gamma.$
First, we have to find an upper bound for $\sup_{x \in \left\lbrace D=\beta \right\rbrace}\dist(x,\left\lbrace D=\alpha \right\rbrace)$. Let $x\in \left\lbrace D=\beta \right\rbrace$, i.e. $D(x)=\beta=\alpha+\eta$. Since $0<\eta \leq \gamma$, $0<D(x)-\alpha \leq \gamma,$ i.e. $x\in \mathcal{K}_\gamma(\alpha)$.
Choose $\lambda:=\frac{2\eta}{m_\nabla} \in
\displaystyle \left (0, \frac{m_\nabla}{M_H} \wedge \min_{1 \leq i \leq N} \frac{r_i}{2}\right)$ %since $\displaystyle \eta \leq  \gamma \leq \frac{m_\nabla}{4}\left(\frac{m_\nabla}{M_H} \wedge \min_{1 \leq i \leq N} \frac{r_i}{2} \right)<m_\nabla \left(\frac{m_\nabla}{M_H} \wedge \min_{1 \leq i \leq N} \frac{r_i}{2} \right)$, 
so that with (\ref{2*}), $$D(y_{-\lambda}) \leq D(x)-\lambda\frac{m_\nabla}{2}=D(x)-\eta=\alpha<D(x).$$ 
%car \alpha+\varepsilon=D(x)
From the above inequality and the continuity property of $z \mapsto D(z)$ on $\mathcal{K}_\varepsilon(\alpha) \supset [y_{-\lambda},x]$, %(since $x \in \mathcal{K}_\gamma(\alpha)$ and $0<\gamma \leq \varepsilon/2$), 
there exists a point $y \in [y_{-\lambda},x]$ s.t. $D(y)=\alpha$. Moreover, $$\| x-y\|\leq \| x-y_{-\lambda}\|=|\lambda|=\frac{2\eta}{m_\nabla}=\frac{2}{m_\nabla}(\beta-\alpha).$$
As a consequence, for all $x \in \left\lbrace D=\beta
\right\rbrace$,
$$\dist(x,\left\lbrace D=\alpha \right\rbrace \leq \|x-y\| \leq \frac{2}{m_\nabla}|\beta-\alpha|.$$
So $$\sup_{x \in \left\lbrace D=\beta \right\rbrace}\dist(x,\left\lbrace D=\alpha \right\rbrace)\leq \frac{2}{m_\nabla}|\beta-\alpha|.$$
In order to get an upper bound for $\sup_{x \in \left\lbrace D=\alpha \right\rbrace}\dist(x,\left\lbrace D=\beta \right\rbrace)$, we use the inequality (\ref{1*}) by proceeding in a similar way.\vspace{0.2cm}\newline
The proof in the case  $0>\beta-\alpha>-\gamma$ is completely analogous.
\end{proof}
\begin{proof}[Proof of Theorem \ref{theoVitessedH}] Let $\alpha \in (0, \alpha_{\max}(P))$.\vspace{0.3cm}\newline
\underline{\textbf{Step 1 : }} we need to find an upper bound for $\sup_{x \in \partial \mathcal{L}(\alpha)}d(x,\partial \mathcal{L}_n(\alpha)).$ \vspace{0.3cm}\newline
Let $x \in \partial \mathcal{L}(\alpha)$. Denote $\varepsilon_n=2 \|D_n-D\|_\infty$. Under the assumptions of Theorem \ref{theoVitessedH} $\varepsilon_n \to 0~\mathbb{P}$-a.s, so that $\mathbb{P}$-a.s, there exists an integer $n_0:=n_0(\omega) \geq 1$ (independent from $x$), s.t. for all $n \geq n_0$, $\varepsilon_n \leq \gamma.$ Taking $\beta= \alpha+\varepsilon_n$, it holds 
\begin{center}
    $\mathbb{P}$-a.s, for all $n\geq n_0$, $d_H(\partial \mathcal{L}(\alpha+\varepsilon_n),\partial \mathcal{L}(\alpha)) \leq A\varepsilon_n.$
\end{center}
Thus, from the above inequality and using the continuity property of $D$, $\mathbb{P}$-a.s, for all $n\geq n_0$, there exists $u_n:=u_{x,\varepsilon_n} \in \partial \mathcal{L}(\alpha+\varepsilon_n)$ i.e. $D(u_n)=\alpha+\varepsilon_n$, and $l_n:=l_{x,\varepsilon_n} \in \partial \mathcal{L}(\alpha-\varepsilon_n)$ i.e. $D(l_n)=\alpha-\varepsilon_n,$ s.t. \begin{equation*}
    \|u_n-x \| \leq A \varepsilon_n \text{~and~} \|l_n-x \| \leq A \varepsilon_n.
\end{equation*}
Let us assume $\|D_n- D\|_{\infty}>0$ (the case $\|D_n- D\|_{\infty}=0$ is trivial). In this case, $$D_n(u_n)=D_n(u_n)+\alpha+\varepsilon_n-D(u_n) \geq \alpha+\varepsilon_n-\|D_n- D\|_{\infty}=\alpha+\|D_n- D\|_{\infty}>\alpha.$$
Similarly, we have $D_n(l_n)<\alpha.$ So $\mathbb{P}$-a.s, for all $n\geq n_0$, $$D_n(l_n)<\alpha<D_n(u_n).$$ For the sake of simplicity, we denote here $\mathcal{L}_n:=\left\lbrace x : D_n(x) \leq \alpha \right\rbrace$. Then, almost surely, for all $n\geq n_0$, $\mathcal{L}_n$ is non-empty (since it contains $l_n$). And by definition, $l_n \in \mathcal{L}_n \subset \overline{\mathcal{L}_n}$. Denoting by $\mathcal{L}_n^c$ the complementary of $\mathcal{L}_n$ in $\mathbb{R}^d$, it holds $u_n \in \mathcal{L}_n^c \subset \overline{\mathcal{L}_n^c}= (\mathring{\mathcal{L}_n})^c,$ that is, $u_n \notin \mathring{\mathcal{L}_n}$. Then, %(using Lemma 3 in \citet{Wills07}) 
\begin{center}
    $\mathbb{P}$-a.s, for all $n\geq n_0$, there exists $z_n \in [l_n,u_n] \cap \partial \mathcal{L}_n$.
\end{center} 
%\begin{figure}[H]
 %   \centering
  %  \input{bordasymptotique}
  %  \caption{Existence of the point $z_n$.}
%\end{figure}
Thus, $\mathbb{P}$-a.s, for all $n\geq n_0$, \begin{align*}
    \dist(x,\partial \mathcal{L}_n(\alpha)) & \leq \|x-z_n\|  \\
    & \leq \|x-u_n\|+ \|u_n-z_n\| \\
    & \leq \|u_n-x\|+\|u_n-l_n\| \\
    & \leq \|u_n-x\|+ \|u_n-x\|+\|x-l_n\| \\
    &\leq 3A\varepsilon_n\\
    &=6A\|D_n-D\|_\infty.
\end{align*}
Since the previous inequality holds for all $x \in \partial \mathcal{L}(\alpha)$, we have, $\mathbb{P}$-a.s, for all $n\geq n_0$, \begin{center}
    $\displaystyle \sup_{x \in \partial \mathcal{L}(\alpha)}d(x,\partial \mathcal{L}_n(\alpha)) \leq 6A\|D_n-D\|_\infty.$
\end{center}
\underline{\textbf{Step 2 : }} Let us find an upper bound for $\sup_{x \in \partial \mathcal{L}_n(\alpha)}d(x,\partial \mathcal{L}(\alpha)).$ \vspace{0.2cm}\newline
Let $x_n \in \partial \mathcal{L}_n(\alpha):=\partial \mathcal{L}_n= \overline{\left\lbrace D_n \leq \alpha \right\rbrace}\cap \overline{ \left\lbrace D_n > \alpha \right\rbrace} \subset \overline{ \left\lbrace D_n \geq \alpha \right\rbrace}=\left\lbrace D_n \geq \alpha \right\rbrace$, since $D_n$ is a.s upper-semicontinuous, so that the upper level set based on $D_n$ is closed. Then, $D_n(x_n) \geq \alpha$. Furthermore, since $x_n \in \overline{\left\lbrace D_n \leq \alpha \right\rbrace}$, there exists $\ell_n$ "close" enough to $x_n$ s.t. $D_n(\ell_n) \leq \alpha$, and s.t. by continuity of $D$, $|D(x_n)-D(\ell_n)| \leq \varepsilon_n/2.$ On the one hand, 
$$ D(x_n)=D_n(x_n)-D_n(x_n)+D(x_n) \geq \alpha -\varepsilon_n/2 \geq \alpha -\varepsilon_n ,$$
on the other hand,
\begin{align*}
    D(x_n) & =D_n(\ell_n)-D_n(\ell_n)+D(\ell_n)-D(\ell_n)+D(x_n) \\
     & \leq \alpha + \varepsilon_n/2+ \varepsilon_n/2,
\end{align*} so, $$|D(x_n)-\alpha| \leq \varepsilon_n. $$
Recall that, a.s for all $n \geq n_0$, $\varepsilon_n \leq \gamma.$ Then, using property \textbf{(L)} with $\beta=D(x_n)$, %($x_n \in \left\lbrace D=D(x_n) \right\rbrace $ trivially),
we can write
 \begin{align*}
   \dist(x_n, \partial \mathcal{L}(\alpha)) \leq d_H(\partial \mathcal{L}(D(x_n)),\partial \mathcal{L}(\alpha)) & \leq A|D(x_n)-\alpha| \\
    & \leq 2 A \|D_n-D\|_\infty.
\end{align*}
Now we deduce that, a.s. for $n$ large enough, $$\sup_{x \in \partial \mathcal{L}_n(\alpha)}d(x,\partial \mathcal{L}(\alpha)) \leq 2A\|D_n-D\|_\infty.$$
Hence the result.
\end{proof}
\textbf{Proofs of Section \ref{MHD section}.}
\begin{proof}[Proof of Proposition \ref{propositiondeMHDD_h}]
\textbf{\textit{(i)}} The function $MHD(\cdot)$ is infinitely differentiable on $\mathbb{R}^d$, and denoting $\mu=\mu_{\bm X}$, we can write for any $1 \leq k \leq d$,\begin{align*}
    \displaystyle \frac{\partial MHD(x)}{\partial x_k}&=-MHD(x)^2  \frac{\partial}{\partial x_k}\Bigg[ \displaystyle \sum_{i,j=1}^d (x_i-\mu_i)(\Sigma_{\bm X}^{-1})_{ij}(x_j-\mu_j) \Bigg]\\
    %&=-MHD(x)^2 \Bigg[ \displaystyle \sum_{j=1}^d (\Sigma_{\bm X}^{-1})_{kj}(x_j-\mu_j) + \sum_{i=1}^d (\Sigma_{\bm X}^{-1})_{ik}(x_i-\mu_i) \Bigg]\\
    &=-MHD(x)^2 \cdot 2 \Bigg[ \displaystyle \sum_{i=1}^d (\Sigma_{\bm X}^{-1})_{ki}(x_i-\mu_i) \Bigg], \text{~}(\Sigma_{\bm X}^{-1}\text{~is symmetric})\\
    &=-2MHD(x)^2\bigg[\Sigma_{\bm X}^{-1}(x-\mu) \bigg]_k.
\end{align*}
So $$(\nabla MHD)_x=-2MHD(x)^2\Sigma_{\bm X}^{-1}(x-\mu).$$
Since $MHD(x)>0$, $$ (\nabla MHD)_x=0 \Longleftrightarrow x=\mu=\mu_{\bm X}.$$
Thus, $$\displaystyle \|(\nabla MHD)_x\|>0,~ \text{for all~} x \neq \mu_{\bm X}.$$
Now since $x \in \mathbb{R}^d \mapsto \| (\nabla MHD)_x \|$ is continuous and $\mathcal{K}_\varepsilon(\alpha)$ is compact, then there exists $x_0 \in \mathcal{K}_\varepsilon(\alpha)$ in which the infimum $m_\nabla$ is attained, $$m_\nabla=\| (\nabla MHD)_{x_0}\|>0.$$ The latter inequality is strict since $x_0 \in \mathcal{K}_\varepsilon(\alpha)$, and $\mu_{\bm X} \notin \mathcal{K}_\varepsilon(\alpha)$ (from the assumption $\varepsilon < 1-\alpha).$ Indeed,
$$\mu_{\bm X} \in \mathcal{K}_\varepsilon(\alpha)  \Leftrightarrow |MHD(\mu_{\bm X})-\alpha| \leq \varepsilon \Leftrightarrow |1-\alpha|=1-\alpha \leq \varepsilon. $$
\textit{\textbf{(ii)}} It suffices to prove \begin{equation} \label{MhdconsisDycker}
    \|MHD_n- MHD\|_{\infty,\mathbb{R}^d} \xrightarrow[n \to \infty]{a.s} 0,
\end{equation} by recalling that $\alpha_{\max}(P)=1$ for $MHD$ depth. The result is hence a straight forward application of Proposition \ref{PropCasalRodriguez} and Theorem \ref{theoVitessedH}. In order to prove (\ref{MhdconsisDycker}), one can refer to the computations in the proof of Theorem \ref{dnRate_MHD} and obtain the desired result knowing that $\hat{\Sigma}_n \xrightarrow[n \to \infty]{a.s} \Sigma,$ and $\hat{\mu}_n \xrightarrow[n \to \infty]{a.s} \mu.$
\end{proof}
\begin{proof}[Proof of Theorem \ref{dnRate_MHD}]
Let $x \in \mathbb{R}^d.$  Denote $\mu:=\mu_{\bm X}$ and $\Sigma:=\Sigma_{\bm X}$, we can write \begin{align*}
    |MHD_n(x)-& MHD(x)|  \\
    &= \displaystyle \left| \frac{1}{1+~^t(x-\hat{\mu}_n)\hat{\Sigma}_n^{-1}(x-\hat{\mu}_n)}-\frac{1}{1+~^t(x-\mu)\Sigma^{-1}(x-\mu)}\right| \\
    & \leq \displaystyle \frac{\left|~^t(x-\hat{\mu}_n)\hat{\Sigma}_n^{-1}(x-\hat{\mu}_n)-~^t(x-\mu)\Sigma^{-1}(x-\mu) \right|}{1+~^t(x-\mu)\Sigma^{-1}(x-\mu)}.
\end{align*}
Since $\Sigma$ is a positive definite symmetric and invertible matrix, we can make the change of variable $y=\Sigma^{-\frac{1}{2}}(x-\mu)$. So that, \begin{align*}
   \displaystyle \| MHD_n & - MHD \|_{\infty, \mathbb{R}^d} 
  % & \leq \displaystyle \sup_{y \in \mathbb{R}^d} \frac{\left|~^t(\Sigma^{\frac{1}{2}}y+\mu-\hat{\mu}_n)\hat{\Sigma}_n^{-1}(\Sigma^{\frac{1}{2}}y+\mu-\hat{\mu}_n)-~^tyy\right|}{1+~^tyy} \\
 \leq \displaystyle \sup_{y \in \mathbb{R}^d} \frac{\left| \| \hat{\Sigma}_n^{-\frac{1}{2}}(\Sigma^{\frac{1}{2}}y+\mu-\hat{\mu}_n)\|^2 -~\|y\|^2\right|}{1+\|y\|^2}.
\end{align*}
Now, denoting by $I_d$ the identity matrix of size $d$, and using a triangle inequality then Cauchy-Schwarz inequality, it holds
\begin{align*}
     &\frac{\left| \| \hat{\Sigma}_n^{-\frac{1}{2}}(\Sigma^{\frac{1}{2}}y+\mu-\hat{\mu}_n)\|^2 -~\|y\|^2\right|}{1+\|y\|^2}\\
     %&=\frac{\left| \| \hat{\Sigma}_n^{-\frac{1}{2}}\Sigma^{\frac{1}{2}}y\|^2+\|\hat{\Sigma}_n^{-\frac{1}{2}}(\mu-\hat{\mu}_n)\|^2+2\left\langle \hat{\Sigma}_n^{-\frac{1}{2}}\Sigma^{\frac{1}{2}}y,\hat{\Sigma}_n^{-\frac{1}{2}}(\mu-\hat{\mu}_n) \right\rangle-~\|y\|^2\right|}{1+\|y\|^2} \\
     &\leq  \frac{\left| \| \hat{\Sigma}_n^{-\frac{1}{2}}\Sigma^{\frac{1}{2}}y\|^2-~\|y\|^2 \right| +\|\hat{\Sigma}_n^{-\frac{1}{2}}(\mu-\hat{\mu}_n)\|^2+2 \left| \left\langle \hat{\Sigma}_n^{-\frac{1}{2}}\Sigma^{\frac{1}{2}}y,\hat{\Sigma}_n^{-\frac{1}{2}}(\mu-\hat{\mu}_n) \right\rangle \right|}{1+\|y\|^2} \\
     & \leq  \frac{\opnorm{\Sigma^{\frac{1}{2}} \hat{\Sigma}_n^{-1} \Sigma^{\frac{1}{2}}- I_d}\|y\|^2 +\|\hat{\Sigma}_n^{-\frac{1}{2}}(\mu-\hat{\mu}_n)\|^2+2 \opnorm{\hat{\Sigma}_n^{-1}} \opnorm{\Sigma^{\frac{1}{2}}}\|y\|\|\mu-\hat{\mu}_n\|}{1+\|y\|^2},
\end{align*}
%since,\\
%$\| \hat{\Sigma}_n^{-\frac{1}{2}}\Sigma^{\frac{1}{2}}y\|^2-~\|y\|^2=\left\langle \hat{\Sigma}_n^{-\frac{1}{2}}\Sigma^{\frac{1}{2}}y,\hat{\Sigma}_n^{-\frac{1}{2}}\Sigma^{\frac{1}{2}}y \right\rangle - \left\langle y,y \right\rangle=\left\langle  (\Sigma^{\frac{1}{2}} \hat{\Sigma}_n^{-1} \Sigma^{\frac{1}{2}}- I_d)y,y \right\rangle$, and
%$\left\langle \hat{\Sigma}_n^{-\frac{1}{2}}\Sigma^{\frac{1}{2}}y,\hat{\Sigma}_n^{-\frac{1}{2}}(\mu-\hat{\mu}_n) \right\rangle=\left\langle \hat{\Sigma}_n^{-1}\Sigma^{\frac{1}{2}}y,\mu-\hat{\mu}_n \right\rangle$, then using Cauchy-Schwarz inequality.
This, together with the fact that $2\|y\|/(1+\|y\|^2) \leq 1$,  for all $y \in \mathbb{R}^d$,
\begin{align*}
   &  \frac{\left| \| \hat{\Sigma}_n^{-\frac{1}{2}}(\Sigma^{\frac{1}{2}}y+\mu-\hat{\mu}_n)\|^2 -~\|y\|^2\right|}{1+\|y\|^2} \\
   & \leq \opnorm{\Sigma^{\frac{1}{2}} \hat{\Sigma}_n^{-1} \Sigma^{\frac{1}{2}}- I_d} +\|\hat{\Sigma}_n^{-\frac{1}{2}}(\mu-\hat{\mu}_n)\|^2+ \opnorm{\hat{\Sigma}_n^{-1}}\cdot \opnorm{\Sigma^{\frac{1}{2}}}\|\mu-\hat{\mu}_n\|.
\end{align*}
Now since the right hand side of the above inequality is independent of $y$, we obtain
\begin{align}\label{eqinfinie} \notag
    & \displaystyle \| MHD_n - MHD \|_{\infty, \mathbb{R}^d} \\ \notag
   % & \leq \opnorm{\Sigma^{\frac{1}{2}} \hat{\Sigma}_n^{-1} \Sigma^{\frac{1}{2}}- I_d} +\|\hat{\Sigma}_n^{-\frac{1}{2}}(\mu-\hat{\mu}_n)\|^2+ \opnorm{\hat{\Sigma}_n^{-1}}\cdot \opnorm{\Sigma^{\frac{1}{2}}}\|\mu-\hat{\mu}_n\| \\ \notag
    & \leq \opnorm{\Sigma^{\frac{1}{2}} \hat{\Sigma}_n^{-1} \Sigma^{\frac{1}{2}}- I_d} + \opnorm{\hat{\Sigma}_n^{-\frac{1}{2}}}^2\|\mu- \hat{\mu}_n\|\left(\|\mu- \hat{\mu}_n\|+\opnorm{\Sigma^{\frac{1}{2}}} \right)\\
    & :=A_n(d),
\end{align}
%since $\opnorm{\hat{\Sigma}_n^{-1}}=\opnorm{\hat{\Sigma}_n^{-\frac{1}{2}}\hat{\Sigma}_n^{-\frac{1}{2}}} \leq \opnorm{\hat{\Sigma}_n^{-\frac{1}{2}}}^2.$\vspace{0.1cm}\newline
The problem reduces to studying the asymptotic behavior of  $A_n(d).$ On one hand, since $\hat{\Sigma}_n \xrightarrow[n \to \infty]{a.s} \Sigma$ and $\hat{\mu}_n \xrightarrow[n \to \infty]{a.s} \mu$, then by the continuity theorem we obtain
 $$\opnorm{\hat{\Sigma}_n^{-\frac{1}{2}}}^2\left(\|\mu- \hat{\mu}_n\|+\opnorm{\Sigma^{\frac{1}{2}}} \right) \xrightarrow[n \to \infty]{\mathbb{P}}\opnorm{\Sigma^{-\frac{1}{2}}}^2\opnorm{\Sigma^{\frac{1}{2}}} >0.$$
Furthermore, by the multivariate Central Limit theorem, it holds that $n^{\frac{1}{2}}(\hat{\mu}_n-\mu) \xrightarrow[n \to \infty]{\mathcal{L}} \mathcal{N}(0,\Sigma)$. Thus (by the continuity theorem and Slutsky's lemma), $\opnorm{\hat{\Sigma}_n^{-\frac{1}{2}}}^2\|\mu- \hat{\mu}_n\|\left(\|\mu- \hat{\mu}_n\|+\opnorm{\Sigma^{\frac{1}{2}}} \right)$ is $\mathcal{O}_{P}\left(n^{-\frac{1}{2}}\right).$
%If $d_n=o(\sqrt{n})$, then by Slutsky's Lemma, $d_n(\hat{\mu}_n-\mu) \xrightarrow[n \to \infty]{\mathcal{L}}0,$ so that $d_n(\hat{\mu}_n-\mu) \xrightarrow[n \to \infty]{\mathbb{P}}0.$ Consequently, \begin{center}
   % $d_n\opnorm{\hat{\Sigma}_n^{-\frac{1}{2}}}^2\|\mu- \hat{\mu}_n\|\left(\|\mu- \hat{\mu}_n\|+\opnorm{\Sigma^{\frac{1}{2}}} \right) \xrightarrow[n \to \infty]{\mathbb{P}}0$.
%\end{center}
On the other hand, to study the first term in $A_n(d)$, we define $$
    F : H \in \mathcal{S}_d(\mathbb{R}) \mapsto \Sigma^{\frac{1}{2}} H^{-1} \Sigma^{\frac{1}{2}},
$$ where $\mathcal{S}_d(\mathbb{R})$ is the vector space of all symmetric real-valued matrices of size $d$. We denote $\mathcal{S}_d^+(\mathbb{R})$ the set of all positive definite symmetric matrices which is an open set in $\mathcal{S}_d(\mathbb{R}).$ Using classical computations of Fréchet differentiable functions, it holds that for all $A \in \mathcal{S}_d^+(\mathbb{R})$, the differential of $F$ at $A$ is given by:  \begin{equation} \label{frechetdiffF} DF_A : H \in \mathcal{S}_d(\mathbb{R}) \mapsto DF_A(H)=-\Sigma^{\frac{
1}{2}}A^{-1}HA^{-1}\Sigma^{\frac{1}{2}}.
\end{equation}
%\begin{lemma}\label{diffferentielledeF}
%For all $A \in \mathcal{S}_d^+(\mathbb{R})$, the differential of F at $A$ is given by : $$ DF_A : H \in \mathcal{S}_d(\mathbb{R}) \mapsto DF_A(H)=-\Sigma^{\frac{
%1}{2}}A^{-1}HA^{-1}\Sigma^{\frac{1}{2}}.$$
%\end{lemma}
%\begin{proof}[Proof of Lemma \ref{diffferentielledeF}]
%Suppose that $\opnorm{H}$ is small enough so that $\opnorm{H} < %\frac{1}{\opnorm{A^{-1}}}$. In this case, we can write \begin{align*}
   % F(A+H)&=\Sigma^{\frac{1}{2}}(A+H)^{-1}\Sigma^{\frac{1}{2}} \\
      %    &=\Sigma^{\frac{1}{2}}(I_d+A^{-1}H)^{-1}A^{-1}\Sigma^{\frac{1}{2}} \\
      %    &=\Sigma^{\frac{1}{2}} \sum_{k=0}^{\infty} (-1)^k(A^{-1}H)^k %A^{-1}\Sigma^{\frac{1}{2}} \\
     %     &=\Sigma^{\frac{1}{2}}A^{-1}\Sigma^{\frac{1}{2}}-\Sigma^{\frac{1}{2}}A^{-1}%HA^{-1}\Sigma^{\frac{1}{2}} + \Sigma^{\frac{1}{2}} \sum_{k=2}^{\infty} %(-1)^k(A^{-1}H)^k A^{-1}\Sigma^{\frac{1}{2}} \\
%          &=F(A)+L(H) +o(\opnorm{H}),
%\end{align*}
%where $L(H)=-\Sigma^{\frac{1}{2}}A^{-1}HA^{-1}\Sigma^{\frac{1}{2}}$ is a linear %application. Therefore, %$DF_A(H)=-\Sigma^{\frac{1}{2}}A^{-1}HA^{-1}\Sigma^{\frac{1}{2}}.$
%\end{proof}
By isomorphism, one can see $\hat{\Sigma}_n$ as an element of $\mathbb{R}^{\frac{d(d+1)}{2}}$. Since $\bm X$ has all of its components in $\mathbb{L}^4$, then a multivariate CLT applies, i.e. there exists $M^* \in \mathcal{S}_{\frac{d(d+1)}{2}}^+(\mathbb{R})$ s.t.
\begin{equation}\label{tclmatricielsigma}
    \sqrt{n}(\hat{\Sigma}_n-\Sigma) \xrightarrow[
n \to \infty]{\mathcal{L}} \mathcal{N}(0,M^*),
\end{equation}
and for notational convenience, the gaussian vector $\mathcal{N}(0,M^*)$ could be rearranged in a size $d$ symmetric random matrix which will be denoted by $E^*.$
For the sake of completeness, we resume a proof of the delta method in the words of \citet{agresti02}  (p. 577). Using a first order Taylor expansion, for all $A \in \mathcal{S}_d^+(\mathbb{R})$ and $X \in \Boule(A,r) \subset \mathcal{S}_d^+(\mathbb{R})$, $r>0$, we can write :
\begin{align*}
    &F(X)-F(A)=DF_A(X-A)+R(X), \text{~with}\\
    &\frac{R(X)}{\opnorm{X-A}} \xrightarrow[X \to A]{}0.
\end{align*}
Then, $\mathbb{P}$-almost surely, using (\ref{frechetdiffF})
\begin{align*}
    F(\hat{\Sigma}_n)-F(\Sigma)&=DF_\Sigma(\hat{\Sigma}_n-\Sigma)+R(\hat{\Sigma}_n), \\
    &=-\Sigma^{\frac{1}{2}}\Sigma^{-1}(\hat{\Sigma}_n-\Sigma)\Sigma^{-1}\Sigma^{\frac{1}{2}} +R(\hat{\Sigma}_n) \\
    &=-\Sigma^{-\frac{1}{2}}(\hat{\Sigma}_n-\Sigma)\Sigma^{-\frac{1}{2}} +R(\hat{\Sigma}_n).
\end{align*}
This, together with (\ref{tclmatricielsigma}), using the continuity theorem and Slutsky's Lemma, we obtain \begin{center}
    $\displaystyle \sqrt{n}R(\hat{\Sigma}_n)=\sqrt{n} \opnorm{\hat{\Sigma}_n-\Sigma}\frac{R(\hat{\Sigma}_n)}{\opnorm{\hat{\Sigma}_n-\Sigma}} \xrightarrow[n \to \infty]{\mathcal{L}}0,$
\end{center} so $\sqrt{n} R(\hat{\Sigma}_n)\xrightarrow[n \to \infty]{\mathbb{P}}0.$ %and with $d_n=o(\sqrt{n})$ we deduce that $d_nR(\hat{\Sigma}_n)\xrightarrow[n \to \infty]{\mathbb{P}}0.$
In addition, from the continuity of $U \mapsto -\Sigma^{-\frac{1}{2}}U\Sigma^{-\frac{1}{2}}$ and using (\ref{tclmatricielsigma}), we obtain  $-\sqrt{n}\Sigma^{-\frac{1}{2}}(\hat{\Sigma}_n-\Sigma)\Sigma^{-\frac{1}{2}} \xrightarrow[n \to \infty]{\mathcal{L}}-\Sigma^{-\frac{1}{2}}E^*\Sigma^{-\frac{1}{2}}.$ Therefore, by continuity of the matrix norm, we deduce that $$\opnorm{\Sigma^{\frac{1}{2}} \hat{\Sigma}_n^{-1} \Sigma^{\frac{1}{2}}- I_d}=\mathcal{O}_{P}\left(n^{-\frac{1}{2}}\right).$$
To conclude, it holds that $A_n(d)=\mathcal{O}_{P}\left(n^{-\frac{1}{2}}\right)$ which implies the desired result:
  $$\| MHD_n - MHD\|_{\infty,\mathbb{R}^d}=\mathcal{O}_{P}\left(n^{-\frac{1}{2}}\right).$$
\end{proof} 
\begin{proof}[Proof of Corollary \ref{MHD_vitesseCCTE}]
Denote $D :=MHD$. Remark that the upper level sets based on $MHD$ are  ellipsoïdes in $\mathbb{R}^d$. It is sufficient to prove that under the assumptions of Corollary \ref{MHD_vitesseCCTE}, Assumption \textbf{(H1)(i)} of Corollary \ref{corSymDiff} is satisfied i.e. $\lambda_d(\mathcal{L}_{n_1}(\alpha)\Delta \mathcal{L}_{n_1}(\alpha)) =\mathcal{O}_{P}\left(v_{n_1}^{-1}\right)$ with $v_{n_1}=n_1^{\frac{1}{2}}$. The result is then a straight forward consequence of Corollary \ref{corSymDiff}. %Notice that when $n_1=n_2=n$ and $\delta=\frac{1}{2r} >0$, the rate of convergence is $d_{n}^{(1-\frac{1}{p})(1-\frac{1}{r})}$ since for $n$ large enough, $d_{n}^{1-\frac{1}{p}} \leq d_n =o(n^{\frac{1}{2}}),$ thus $d_{n}^{(1-\frac{1}{p})(1-\frac{1}{r})}=o(n^{\frac{1}{2}(1-\frac{1}{r})})$. 
We introduce : $$\ell_{n_1}=\ell_{n_1}(\alpha):=d_H(\partial \mathcal{L}_{n_1}(\alpha), \partial \mathcal{L}_D(\alpha)),$$ and the tube around the boundary $\partial \mathcal{L}_D(\alpha)$ of radius $\ell_{n_1}$ defined by
\begin{align*}
    \Tube(\partial \mathcal{L}_D(\alpha), \ell_{n_1}):=\left\lbrace z \in \mathbb{R}^d : \dist(z,\partial \mathcal{L}_D(\alpha)) \leq \ell_{n_1} \right\rbrace.
\end{align*}
Since $\partial \mathcal{L}_D(\alpha)$ is closed and $\ell_{n_1}$ is small enough ($\mathbb{P}$-a.s. for $n_1$ large enough, according to Proposition \ref{propositiondeMHDD_h} for $MHD$), then $\mathbb{P}$-a.s, (cf. \citet{Weyl39}, p. 461)
\begin{align*}
    \lambda_d(\mathcal{L}_{n_1}(\alpha)\Delta \mathcal{L}_D(\alpha))& \leq \lambda_d \left[\Tube(\partial \mathcal{L}_D(\alpha),\ell_{n_1}) \right] \\
    & \approx A_{d-1}(\alpha)\ell_{n_1} \text{~~for large~} n_1, ~ \text{(Weyl)} \\
    & \leq A_{d-1}(\alpha) \cdot C \| MHD_{n_1} -MHD\|_{\infty,\mathbb{R}^d} \text{~~for large~}n_1,
\end{align*}
and the above inequality is obtained by Proposition \ref{propositiondeMHDD_h}, where $ A_{d-1}(\alpha)$ is the $(d-1)$-dimensional volume of $\mathcal{L}_D(\alpha)$.
\begin{figure}[H]
    \centering
    \input{difference_symetrique}
    \caption{Illustration of $\lambda_d(\mathcal{L}_{n_1}(\alpha)\Delta \mathcal{L}_D(\alpha))$ (yellow), and the tube around $\mathcal{L}_D(\alpha)$ of radius $\ell_{n_1}$ (blue).}
\end{figure}
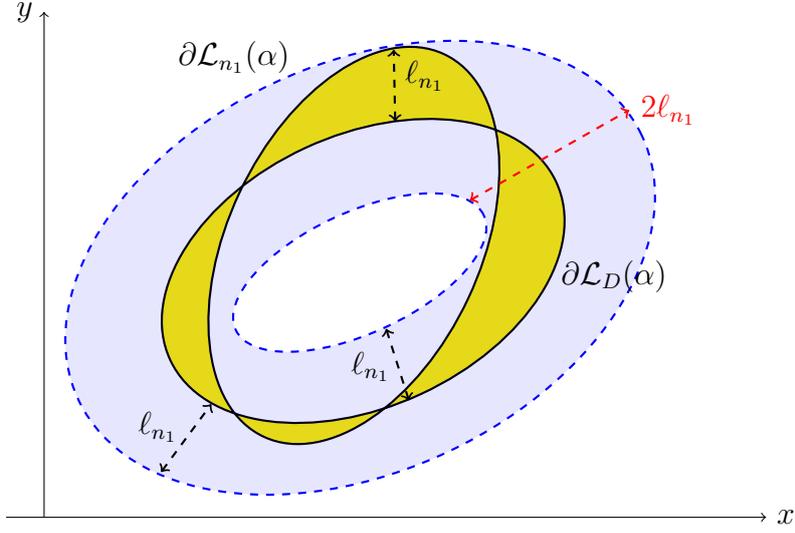
Finally, according to Theorem \ref{dnRate_MHD}, $\| MHD_{n_1} -MHD\|_{\infty,\mathbb{R}^d} = \mathcal{O}_{P}\left(n^{-\frac{1}{2}}\right).$ Hence, $$\lambda_d(\mathcal{L}_{n_1}(\alpha)\Delta \mathcal{L}_D(\alpha))=\mathcal{O}_{P}\left(n^{-\frac{1}{2}}\right).$$
\end{proof}
\newpage
\bibliographystyle{apalike} 
\bibliography{mabiblio} 
\end{document}

%% file: difference_symetrique.tex
\begin{tikzpicture}
    \def\firstellipse{[rotate=25](0.6,2) ellipse (2.8cm and 1.8cm)}
    \def\secondellipse{[rotate=40] (2,1.4) ellipse (2.8cm and 1.66cm)}
    
    \draw[->] (-4.5,-1.2) -- (-4.5,5.5) node[left] {$y$};
    \draw[->] (-5,-1.2) -- (5,-1.2) node[right] {$x$};
      
    \fill[even odd rule, yellow] \firstellipse \secondellipse;
    \draw[thick] \firstellipse \secondellipse;
    %\draw (0.1, 4.3) node { \small $\mathcal{L}_n(\alpha) \bigtriangleup \mathcal{L}_D(\alpha)$};
    \draw[<->][dashed,thick] (0.12, 4.05) -- (0.1, 5) node[below right] {$\ell_{n_1}$};
   
    \def\thirdellipse{[rotate=25,dashed,thick,blue](0.55,2) ellipse (1.8cm and 0.8cm)};
    \def\fourthellipse{[rotate=0.5,dashed,thick,blue](0.6,2.05) ellipse (4.1cm and 2.7cm)};
    \fill[even odd rule, blue,opacity=0.1] \thirdellipse \fourthellipse;
    \draw[thick] \thirdellipse \fourthellipse;
    
    \draw[<->][dashed,thick] (0, 1.3) -- (0.3, 0.35);
    \draw (-0.2,0.8) node {$\ell_{n_1}$};
    \draw[<->][dashed,thick,red] (1.1, 3) -- (3.2, 4.2) node[right] {$2\ell_{n_1}$};
    \draw[<->][dashed,thick] (-2.96, -0.6) -- (-2.3, 0.3);
    \draw (-3,0) node {$\ell_{n_1}$};
   % \filldraw[blue,opacity=0.1] (2.5,1.75) rectangle (3.8,2.25);
    \draw  (3, 2)  node { $\partial \mathcal{L}_D(\alpha)$};
    %\filldraw[white] (-2.5,3.75) rectangle (-1,4);
    \draw  (-2, 4.9)  node { $\partial \mathcal{L}_{n_1}(\alpha)$};
%\node[fill = white, shape = circle] (0,0) {$X$};
\end{tikzpicture}

%% file: main.bbl
\begin{thebibliography}{}

\bibitem[Agresti, 2002]{agresti02}
Agresti, A. (2002).
\newblock {\em Categorical data analysis.}
\newblock NJ: John Wiley \& Sons, Inc.

\bibitem[Belzunce et~al., 2007]{belzunce2007}
Belzunce, F., Casta{\~n}o, A., Olvera-Cervantes, A., and Su{\'a}rez-Llorens, A.
  (2007).
\newblock Quantile curves and dependence structure for bivariate distributions.
\newblock {\em Computational Statistics \& Data Analysis}, 51(10):5112--5129.

\bibitem[Beran and Millar, 1997]{Beran1997}
Beran, R.~J. and Millar, P.~W. (1997).
\newblock {\em Multivariate Symmetry Models. In Festschrift for Lucien Le Cam:
  Research Papers in Probability and Statistics}, pages 13--42.
\newblock Pollard, David and Torgersen, Erik and Yang, Grace L. Springer New
  York.

\bibitem[Cousin and Di~Bernardino, 2013]{cousin2013}
Cousin, A. and Di~Bernardino, E. (2013).
\newblock On multivariate extensions of value-at-risk.
\newblock {\em Journal of multivariate analysis}, 119:32--46.

\bibitem[Cuevas et~al., 2006]{CuevasRodri06}
Cuevas, A., González-Manteiga, W., and Rodríguez–Casal, A. (2006).
\newblock Plug-in estimation of general level sets.
\newblock {\em Australian \& New Zealand J. Statist.}, 48:7--19.

\bibitem[Dehaan and Huang, 1995]{dehaan1995}
Dehaan, L. and Huang, X. (1995).
\newblock Large quantile estimation in a multivariate setting.
\newblock {\em Journal of Multivariate Analysis}, 53(2):247--263.

\bibitem[Denuit et~al., 2006]{denuit2006actuarial}
Denuit, M., Dhaene, J., Goovaerts, M., and Kaas, R. (2006).
\newblock {\em Actuarial theory for dependent risks: measures, orders and
  models}.
\newblock John Wiley \& Sons.

\bibitem[Di~Bernardino et~al., 2013]{DBLMDP13}
Di~Bernardino, E., Laloë, T., Maume-Deschamps, V., and Prieur, C. (2013).
\newblock Plug-in estimation of level sets in a non-compact setting with
  applications in multivariate risk theory.
\newblock {\em ESAIM: Probability and Statistics}, 17.

\bibitem[Di~Bernardino et~al., 2015]{DBLS2015}
Di~Bernardino, E., Laloë, T., and Servien, R. (2015).
\newblock Estimating covariate functions associated to multivariate risks: a
  level set approach.
\newblock {\em Metrika, Springer Verlag}, pages 497--526.

\bibitem[Dyckerhoff, 2016]{dyckerhoff2017}
Dyckerhoff, R. (2016).
\newblock Convergence of depths and depth-trimmed regions.
\newblock {\em arXiv preprint arXiv:1611.08721}.

\bibitem[Eberlein et~al., 2007]{Eberlein07}
Eberlein, E., Frey, R., Kalkbrener, M., and Overbeck, L. (2007).
\newblock Mathematics in financial risk management.
\newblock {\em Jahresbericht der DMV}.

\bibitem[Liu, 1990]{liu1990}
Liu, R. (1990).
\newblock On a notion of data depth based on random simplices.
\newblock {\em The Annals of Statistics}, pages 405--414.

\bibitem[Rodríguez-Casal, 2003]{RodCasal03}
Rodríguez-Casal, A. (2003).
\newblock Estimacíon de conjuntos y sus fronteras. un enfoque geometrico.
\newblock {\em PhD thesis, University of Santiago de Compostela}.

\bibitem[Torres et~al., 2020]{TRDiB20}
Torres, R., Di~Bernardino, E., Laniado, H., and Lillo, R. (2020).
\newblock On the estimation of extreme directional multivariate quantiles.
\newblock {\em Communications in Statistics-Theory and Methods},
  49(22):5504--5534.

\bibitem[Weyl, 1939]{Weyl39}
Weyl, H. (1939).
\newblock On the volume of tubes.
\newblock {\em American Journal of Mathematics}, Vol. 61:461--472.

\bibitem[Zuo and Serfling, 2000]{ZS2000}
Zuo, Y. and Serfling, R. (2000).
\newblock General notions of statistical depth function.
\newblock {\em Annals of statistics}, pages 461--482.

\end{thebibliography}
